\documentclass[review,onefignum,onetabnum]{siamart190516}

\usepackage{lipsum}
\usepackage{amsmath,amssymb,amsfonts,latexsym,stmaryrd}
\usepackage{graphicx,float}
\usepackage{mathrsfs} 
\usepackage{color}
\usepackage{epstopdf}
\usepackage{algorithmic}
\usepackage{multirow}
\ifpdf
  \DeclareGraphicsExtensions{.eps,.pdf,.png,.jpg}
\else
  \DeclareGraphicsExtensions{.eps}
\fi

\def\O{\Omega}

\def\e{\varepsilon}
\def\g{\gamma}

\def\l{\lambda}

\def\E{E}
\def\VK{V^{\E}}

\renewcommand\sp{\mathop{\mathrm{Sp}}\nolimits}

\newcommand{\jump}[1]{\llbracket #1 \rrbracket}

\newcommand{\n}{\boldsymbol{n}}
\usepackage{booktabs}
\usepackage{float}

\newcommand\bu{\boldsymbol{u}}
\newcommand\bv{\boldsymbol{v}}

\newcommand\bn{\boldsymbol{n}}

\def\Vh{V_h}

\def\PiK{\Pi^{\E}}
\def\CM{\mathcal{X}}
\def\CN{\mathcal{Y}}
\def\CE{\mathcal{E}}
\renewcommand\sp{\mathop{\mathrm{sp}}\nolimits}

\def\CT{\mathcal{T}}
\def\CM{\mathcal{X}}
\def\CN{\mathcal{Y}}

\def\O{\Omega}
\def\P{\mathbb{P}}
\def\PiK{\Pi^{\nabla,E}}

\def\Pio{\Pi^{E}}
\def\Vh{V_h}
\def\VK{V^{E}_h}
\def\WK{\widetilde{V}_h^E}
\def\l{\lambda}
\def\g{\gamma}

\renewcommand\div{\mathop{\mathrm{div}}\nolimits}
\renewcommand\div{\mathop{\mathrm{div}}\nolimits}

\def\CT{{\mathcal T}}

\newcommand\R{\mathbb{R}}

\renewcommand\H{\mathrm{H}}
\renewcommand\L{\mathrm{L}}

\renewcommand\O{\Omega}

\renewcommand\div{\mathop{\mathrm{div}}\nolimits}

\renewcommand\sp{\mathop{\mathrm{sp}}\nolimits}

\newcommand{\vertiii}[1]{{\left\vert\kern-0.25ex\left\vert\kern-0.25ex\left\vert #1 
    \right\vert\kern-0.25ex\right\vert\kern-0.25ex\right\vert}}

\nolinenumbers

\newsiamremark{remark}{Remark}
\newsiamremark{hypothesis}{Hypothesis}
\crefname{hypothesis}{Hypothesis}{Hypotheses}
\newsiamthm{claim}{Claim}

\headers{Non-conforming VEM for the acoustic problem}{D. Amigo, F. Lepe and  G. Rivera}

\title{Error analysis for a non-conforming virtual element  discretization of the acoustic problem\thanks{Submitted to the editors DATE.
\funding{DA and FL were partially supported by DIUBB through project 2120173 GI/C Universidad del B\'io-B\'io.  GR was supported by Universidad de Los Lagos Regular R02/21 and ANID-Chile through FONDECYT project 1231619 (Chile).
}}}

\author{Danilo Amigo\thanks{GIMNAP-Departamento de Matem\'atica, Universidad del B\'io-B\'io, Casilla 5-C, Concepci\'on, Chile. 
\texttt{danilo.amigo2101@alumnos.ubiobio.cl}.}
\and
Felipe Lepe\thanks{GIMNAP-Departamento de Matem\'atica, Universidad del B\'io-B\'io, Casilla 5-C, Concepci\'on, Chile. 
\texttt{flepe@ubiobio.cl}.}
\and
Gonzalo Rivera\thanks{Departamento de Ciencias Exactas, Universidad de Los Lagos, Osorno, Chile.
\texttt{gonzalo.rivera@ulagos.cl}}}

\usepackage{amsopn}

\ifpdf
\hypersetup{
  pdftitle={VEM allowing small edges for the acoustic problem},
  pdfauthor={Danilo Amigo, Felipe Lepe, Gonzalo Rivera}
}
\fi

\begin{document}
\maketitle
\nolinenumbers

\begin{abstract}
We introduce non conforming virtual elements to approximate the eigenvalues and eigenfunctions of the two dimensional acoustic vibration problem. 
We  focus our attention on   the pressure formulation of the acoustic vibration problem in order to discretize it with a  suitable non conforming virtual space for $\H^1$.
With the aid of the theory of non-compact operators we prove convergence and spectral correctness of the method. To illustrate the theoretical results, we 
report numerical tests on different polygonal meshes, in order to show the accuracy of the method on the approximation of the spectrum.
\end{abstract}



\section{Introduction}
\label{sec:introduccion}
Let $\O$ be an open and bounded bidimensional domain with Lipschitz boundary $\partial\O$. The acoustic vibration problem is:  Find  the natural frequency $\omega\in\mathbb{R}$, the displacement $\bu$, and the pressure $p$ on a domain $\Omega\subset\mathbb{R}^{\texttt{d}}$, such that 
\begin{equation}\label{def:acustica}
\left\{
\begin{array}{rcll}
\nabla p-\omega^2\rho \bu & = & \boldsymbol{0}&\text{in}\,\O\\
p+\rho c^2\div\bu & = & 0 &\text{in}\,\O\\
\bu\cdot\boldsymbol{n}&=&0&\text{on}\,\partial\O,
\end{array}
\right.
\end{equation}
where $\rho$ is the density, $c$ is the sound speed, and $\boldsymbol{n}$ is the outward unitary vector. On this case, we are assuming that the fluid in inviscid and hence, the eigenvalues are real. This problem has paid the attention of engineers  and mathematicians due to the large number of applications of such a system in different fields, as the design of noise reduction devices, design of ships, aircrafts, bridges, buildings, etc. Here we mention \cite{MR1993937,MR1770352, MR1342293, HYBi2017NonConf}.

Regarding  the importance of solving  \eqref{def:acustica}, several numerical approaches have emerged revealing the necessity to solve properly the acoustic vibration problem. On this sense, and taking in consideration not only the numerical method but also the context in which the acoustic problem is applicable we mention \cite{MR1770352,MR1342293,MR1993937,MR2086168}  as main references. These references made mention of the fact that the acoustic problem is important to be analyzed when the system is coupled with other media as elastic structures, the design of noise reduction devices, the control of noise, etc. In fact, our studies of  the VEM for the  acoustic equations are in that direction, since precisely the pure pressure formulation of associated to \eqref{def:acustica} allows us to include $\H^1$ spaces that can be coupled with an elastic structure safely on the interface as in \cite{MR1993937}, allowing the presence of small edges  precisely on the contact interface, without  the need to incorporate additional hypotheses on the VEM. Let us mention \cite{amigo2023vem,MR4658607} as part of this research topic.

Regarding to eigenvalue problems, these have been studied with the virtual element method on different contexts and is an ongoing research topic, proving the relevance of the method. On this subject we recall 
\cite{MR4658607,MR3867390,MR4229296,MR4284360,MR3340705, MR3867390, MR4050542}. Particularly for the acoustic problem we refer to \cite{BMRR,MR4550402} and for NCVEM applied to eigenvalue problems we have \cite{GMV2018} which is, for the best of our knowledge, the first work where NCVEM is applied to eigenvalue problems, inspired by the pioneer work of \cite{MR3507277}.

Let us focus on  \cite{GMV2018}. Here the authors have proved, for the Laplace eigenvalue problem with null Dirichlet boundary condition that the NCVEM is convergent and spurious free. These are  of course the desirable features of any numerical method to approximate eigenvalue problems. Moreover, the analysis is performed under the approach of the compact operators theory of \cite{BO}. In our case, despite to the fact that we will avoid the displacement of the fluid leading to a Laplace eigenvalue problem, according to \eqref{def:acustica} our boundary condition is different and hence, our solution operators, continuous and discrete,  must be different compared with those  on  \cite{GMV2018}. This fact will introduce a substantial difference with our paper, since in our case we need to employ the non-compact theory of operators \cite{DNR1,DNR2} in order to derive convergence and error estimates for our method due to the non-conformity of our VEM spaces. It is important to note that all works involving non-conforming virtual element method for spectral problems available in the literature use classical compact operator theory (see \cite{ADAK,GMV2018} for instance), so this is the first work in which non-compact operator theory would be used.


The article is organized in the following way: in Section \ref{sec:model} we present the acoustic model problem written in terms of the pressure. We present the variational formulation, the solution operator, regularity results and the corresponding spectral characterization. The core of our paper begins in Section  \ref{sec:virtual} where we introduce the NCVEM. This implies the assumptions on the polygonal meshes, elements of the mesh, jumps, local and global virtual spaces and their degrees of freedom and hence, the discrete bilinear forms that allows us to define the discrete eigenvalue problem. These tools lead to the analysis of error estimates for the eigenvalues and eigenfunctions which we derive according to the non-conforming nature of the proposed method. Finally in Section \ref{sec:numerics} we report a series of numerical examples to assess the performance of the method on different domains and polygonal meshes. 

 \section{The  model problem}
 \label{sec:model}
\label{sec:pressure}
The intention of this paper is to study system \eqref{def:acustica} in the most simple way in order to use the NCVEM of our interest. To do this task, using the second equation of \eqref{def:acustica} we are able to eliminate the displacement from \eqref{def:acustica}  leading to the following problem; find the pressure $p$ and the frequency $\omega$ such that
\begin{equation}\label{def:acustica_pressure}
\left\{
\begin{aligned}
 c^2\div \left(\frac{1}{\rho}\nabla p\right)+\frac{\omega^2}{\rho} p& = &0&\text{ in }\,\O,\\
\nabla p\cdot\boldsymbol{n}&=&0&\text{ on }\,\partial\O,
\end{aligned}
\right.
\end{equation}

 A variational formulation for \eqref{def:acustica_pressure} is: Find $\omega\in\mathbb{R}$ and $0\neq p\in \H^1(\O)$ such that
\begin{equation*}
\displaystyle c^2\int_{\O}\frac{1}{\rho}\nabla p\cdot\nabla v=\omega^2\int_{\O}\frac{1}{\rho}pv \quad\forall v\in \H^1(\O).
\end{equation*}
Let us define the bilinear forms $a:\H^1(\O)\times \H^1(\O)\rightarrow\mathbb{R}$ and $b:\H^1(\O)\times \H^1(\O)\rightarrow\mathbb{R}$, which are given by
\begin{equation*}
a(q,v):=c^2\int_{\O}\frac{1}{\rho}\nabla q\cdot\nabla v\quad \text{and}\quad b(q,v):=\int_{\O}\frac{1}{\rho}qv\quad\forall q,v \in\H^1(\O).
\end{equation*}
With a shift argument and setting $\lambda := \omega^{2}+1$, we arrive to the following problem: Find $\lambda\in\mathbb{R}$ and $0\neq p\in \H^1(\O)$ such that
\begin{equation}
\label{eq:pression}
\widehat{a}(p,v)=\lambda b(p,v)\quad\forall v\in \H^1(\O),
\end{equation}
where the bilinear form $\widehat{a}:\H^1(\O)\times \H^1(\O)\rightarrow\mathbb{R}$ is defined for all $q,v\in \H^1(\O)$ by 
$$ \widehat{a}(q,v) := a(q,v) + b(q,v).$$

It is easy to check that $\widehat{a}(\cdot,\cdot)$ is coercive in $\H^1(\O)$. This allow to us to introduce the solution operator $T: \L^2(\O)\rightarrow \H^1(\O)$, defined by 
$Tf=\widehat{p}$, where $\widehat{p}\in \H^1(\O)$ is the solution of the corresponding  associated source problem: Find $\widehat{p}\in\H^1(\O)$ such that
\begin{equation}
\label{eq:source_pr}
\widehat{a}(\widehat{p},v)=b(f,v)\quad\forall v\in\H^1(\O).
\end{equation}

The regularity results that we need for our purposes are the ones related to the Laplace problem with pure null boundary conditions. This regularity is stated in the following lemma (see \cite[Lemma 2.2]{MR3340705} and \cite{MR0775683}).
\begin{lemma}
\label{lmm:regularity}
There exists $r_{\Omega}>1/2$ such that the following results hold:
\begin{enumerate}
\item For all $f\in\H^1(\O)$ and for all $r\in (1/2,r_{\O})$ the solution $\widehat{p}$ of \eqref{eq:source_pr} satisfies $\widehat{p}\in\H^{1+s}(\O)$ with $s:=\min\{r,1\}$. Moreover, there exists a constant $C>0$ such that
\begin{equation*}
\|\widehat{p}\|_{1+s,\O}\leq C\|f\|_{1,\O};
\end{equation*}
\item If $p$ is an eigenfunction of problem \eqref{eq:pression} with eigenvalue $\lambda$, for all $r\in(1/2,r_{\O})$ there hold that $p\in\H^{1+r}(\O)$ and also, there exists a constant $C>0$, depending on $\lambda$, such that
\begin{equation*}
\| p\|_{1+r,\O}\leq C\|p\|_{1,\O}.
\end{equation*}
\end{enumerate} 
\end{lemma}

In virtue of Lemma \ref{lmm:regularity}, the solution operator $T$  results to be compact due to the compact inclusion of $\H^{1+s}(\O)$ onto $\H^1(\O)$ and self-adjoint with respect to $\widehat{a}(\cdot,\cdot)$. We observe that $(\lambda,p) \in \mathbb{R}\times \H^1(\O)$ solves \eqref{eq:pression} if and only if $(\mu,p) \in \mathbb{R}\times \H^1(\O)$ is an eigenpair of $T$, with $\mu := 1/\lambda$. Finally, since we have the additional regularity for the eigenfunctions, the following spectral characterization of $T$ holds.
\begin{lemma}[Spectral Characterization of $T$]
The spectrum of $T$ satisfies $\sp(T) = \{0,1\} \cup \{\mu_k\}_{k\in\mathbb{N}}$, where $\{\mu_k\}_{k\in\mathbb{N}}$ is a sequence of real and positive eigenvalues that converge to zero, according to their respective multiplicities.
\end{lemma}


\section{The virtual element method}
\label{sec:virtual}
Let us now introduce the ingredients to establish the virtual element method  for the eigenvalue problem \eqref{eq:pression}. First we recall the mesh construction and the assumptions considered in \cite{BBCMMR2013} for the virtual element
method. 
Let $\left\{\CT_h\right\}_h$ be a sequence of decompositions of $\Omega$ into polygons which we denote by $E$ and let $\CE_h$ denote the skeleton of the partition. By $\CE_h^o$ and $\CE_h^\partial$ we will refer to the set of interior and boundary edges, respectively. Let us denote by $h_E$   the diameter of the element $E$ and $h$ the maximum of the diameters of all the elements of the mesh, i.e., $h:=\max_{E\in\Omega}h_E$.  Moreover, for simplicity, in what follows we assume that $\rho $ is  piecewise constant with  respect to the decomposition $\mathcal{T}_h$, i.e., they are  piecewise constants for all $E\in \mathcal{T}_h$ (see for instance \cite{BLR2017}).

 For the analysis of the VEM, we will make as in \cite{BBCMMR2013} the following
assumptions: 
\begin{itemize}
\item \textbf{A1.} There exists $\g>0$ such that, for all meshes
$\CT_h$, each polygon $\E\in\CT_h$ is star-shaped with respect to a ball
of radius greater than or equal to $\g h_{\E}$.
\item \textbf{A2.} The distance between any two vertexes of $\E$ is $\geq Ch_\E$, where $C$ is a positive constant.
\end{itemize}

For any simple polygon $ E$ we define 
\begin{align*}
	\widetilde{V}_h^E:=\left\{v_h\in \H^1(E):\Delta v_h \in \mathbb{P}_k(E),
	\dfrac{\partial v_h}{\partial\bn }\in \mathbb{P}_{k-1}(e) \  \forall e \in \partial E \right\}.
\end{align*}

Now, in order to choose the degrees of freedom for $\widetilde{V}_h^E$ we define
\begin{itemize}
	\item $\mathcal{D}^1$: the moments of $v_h$ of order up to $k-1$ on each edge $e\in\partial E$:
	$$\dfrac{1}{|e|}\int_e v_hp_{k-1} \text{ for all } p_{k-1}\in \P_{k-1}(e) \text{ and for all }e\in\partial E;$$
	\item  $\mathcal{D}^2$: the moments of $v_h$ of order up to $k-2$ on $E$:
	$$\dfrac{1}{|E|}\int_E v_h p_{k-2} \text{ for all }p_{k-2}\in \P_{k-2}(E).$$
\end{itemize}
as a set of linear operators from $\widetilde{V}_h^E$ into $\R$. In \cite{AABMR13} it was established that $\mathcal{D}^{1}$ and $\mathcal{D}^{2}$ constitutes a set of degrees of freedom for the space $\widetilde{V}_h^E$.

 On the other hand, we define  the projector $\PiK_k:\ \WK\longrightarrow\P_k(E)\subseteq\WK$ for
each $v_{h}\in\WK$ as the solution of 
\begin{equation*}
\int_E (\nabla\PiK_k v_{h}-\nabla v_{h})\cdot\nabla q_{k}=0
\quad\forall q_{k}\in\P_k(E),\qquad \overline{\PiK_k v_{h}}=\overline{v_{h}},
\end{equation*}
where  for any sufficiently regular
function $v$, we set $\overline{v}:=\vert\partial E\vert^{-1}\left(v,1\right)_{0,\partial E}.$
We observe that the term $\PiK_k v_{h}$ is well defined and computable from the degrees of freedom  of $v_h$ given by $\mathcal{D}^1$ and $\mathcal{D}^{2}$. In addition, the projector $\PiK_k$ satisfies the identity  $\PiK_k(\P_{k}(E))=\P_{k}(E)$ (see for instance \cite{AABMR13}).

We are now in position  to introduce our local virtual space
\begin{equation*}\label{Vk}
\VK
:=\left\{v_{h}\in 
\WK: \displaystyle \int_E \PiK_k v_{h}p_{k}=\displaystyle \int_E v_{h}p_{k},\quad \forall p_{k}\in 
\mathbb{P}_k(E)\setminus \P_{k-2}(E)\right\}.
\end{equation*}

Since $\VK\subset \WK$, the operator $\PiK$ is well defined on $\VK$ and computable  only on the basis of the output values of the operators in $V_h^E$.
In addition, due to the particular property appearing in definition of the space $\VK$, it can be seen that $\forall p_{k} \in \mathbb{P}_k(E)$ and $\forall v_h\in \VK$ the term $(v_h,p_{k})_{0,E}$
is computable from $\PiK_k v_h$, and hence  the  ${\mathrm L}^2(E)$-projector operator $\Pio_k: \VK\to \P_k(E)$ defined  by
$$\int_E \Pio_kv_h p_{k}=\int_E v_h p_{k}\qquad \forall p_{k}\in \mathbb{P}_k(E),$$
depends only on the values of the degrees of freedom of $v_h$ and hence is computable only on the basis of the output values of the operators $\mathcal{D}^{1}$ and $\mathcal{D}^2$. 

Now, for any $s>0$, we  introduce the broken Sobolev space of order $s$
\begin{equation*}\H^s(\CT_h):=\displaystyle\prod_{E\in\CT_h}\H^s(E)=\left\{v\in \L^2(\O): v|_E\in\H^s(E)\right\},
\end{equation*}
endowed with  the broken $\H^s$-norm:
$$\|v\|_{s,h}^2=\sum_{E\in\CT_h}\|v\|_{s,E}^2\qquad \forall v\in \H^s(\CT_h),$$
and for $s=1$ the broken $\H^1$-seminorm
$$|v|_{1,h}^2=\sum_{E\in\CT_h}\|\nabla v\|_{0,E}^2\qquad \forall v\in \H^1(\CT_h).$$

Let $e\subset \partial E^+\cap\partial E^-$ be the internal side shared by elements $E^+_e$ and $E^-_e$, and $v$ a function that belongs to $\H^1(\CT_h)$. We denote the traces of $v$ on $e$ from the interior of elements $E^{\pm}_e$ by $v^{\pm}_e$, and the unit normal vectors to $e$ pointing from $E^{\pm}_e$ to $E^{\mp}_e$ by $\n^{\pm}_e$. Then, we introduce the jump operator $\jump{v}= v^+_e\n^+_e+v^-_e\n^-_e$ at each internal side $\e\in \CE^0$ and $\jump{v} = v_e\n_e$ at each boundary side $e\in \CE^{\partial}$.

It is practical to introduce a subspace $\H^1(\CT_h)$ that incorporates inherent continuity. For $k\geq 1$, we define
\begin{equation*}\label{eq_subsetcont}
\H^{1,nc}(\CT_h,k):=\left\{v\in \H^1(\CT_h):\int_e\jump{v}\cdot\n_eq_{k-1}=0\;\ \forall q_{k-1}\in\P_{k-1}(e),\,\ \forall e\in \CE_h\right\}.
\end{equation*} 

Finally, for every decomposition $\CT_h$ of $\Omega$ into simple polygons $ E$ we define the global virtual space
\begin{equation}
\label{eq:globa_space}
\Vh:=\left\{v_h\in \H^{1,nc}(\CT_h,k):\ v_h|_{ E}\in\VK\quad\forall E\in\CT_h\right\}.
\end{equation}
\subsection{Discrete bilinear forms}
In order to propose the discrete counterparts of $a(\cdot,\cdot)$ and $b(\cdot,\cdot)$, we split these  forms as follows
\begin{equation*}
a(q,v) = \displaystyle{\sum_{E \in \mathcal{T}_{h}} a^{E}(q,v)}, \quad b(q,v) = \displaystyle{\sum_{E \in \mathcal{T}_{h}} b^{E}(q,v)}.
\end{equation*}
First, we consider any symmetric positive definite bilinear form $S^{\E}: \VK \times \VK \rightarrow \mathbb{R}$ that satisfies
\begin{equation}\label{eq:stability}
c_0 a^{\E}(q_h,q_h)\leq S^{\E}(q_h,q_h)\leq c_1a^{\E}(q_h,q_h)\qquad \forall q_h\in\VK \cap \text{ker}(\PiK),
\end{equation}
with $c_0$ and $c_1$ being positive constants depending on the mesh assumptions. Then, we define
\begin{equation*}
\label{ec:ahch}
a_h(q_h,v_h)
:=\sum_{\E\in\CT_h}a_h^{\E}(q_h,v_h),
\end{equation*}
where $a_h^{\E}(\cdot,\cdot)$ is the bilinear form defined from
$\VK\times\VK$ onto $\mathbb{R}$ by
\begin{equation*}
\label{21}
a_h^{\E}(q_{h},q_{h})
:=a^{\E}\big(\PiK_k q_h,\PiK_k v_h\big)
+S^{\E}\big(q_h-\PiK_k q_h,v_h-\PiK_k v_h\big),
\end{equation*}
for every $q_h,v_h\in\VK$.
Moreover, it is easy to check that $a_h^{\E}(\cdot,\cdot)$ must admit the following properties:
\
\begin{itemize}
\item Consistency: For all $q_h\in\mathbb{P}_k(\E)$ and $v_h\in\VK$
\begin{equation*}
a_h^{\E}(q_h,v_h)=a^{\E}(q_h,v_h),
\end{equation*}
\item Stability: There exist two positive constants $\alpha_*$ and $\alpha^*$, independent of $h$ and $\E$, such that
\begin{equation*}
\alpha_*a^{\E}(q_h,q_h)\leq a_h^{\E}(q_h,q_h)\leq\alpha^*a^{\E}(q_h,q_h)\quad\forall q_h\in\VK.
\end{equation*}
\end{itemize}

On the other hand, to introduce the local discrete counterpart of $b^{E}(q_{h},v_{h})$, we consider any symmetric and semi-positive definite bilinear form $S_{0}^{E} : V_{h}^{E}\times V_{h}^{E} \rightarrow \mathbb{R}$ satisfying
\begin{equation*}
b_{0}b^{E}(v_{h},v_{h}) \leq S_{0}^{E}(v_{h},v_{h}) \leq b^{1}b^{E}(v_{h},v_{h}) \quad \forall v_{h} \in V_{h}^{E},
\end{equation*}
where $b_{0}$ and $b^{1}$ are two   positive constants. Then, we define for each polygon $E$ the local (and computable) bilinear form $b_{h}^{E} : V_{h}^{E}\times V_{h}^{E} \rightarrow \mathbb{R}$ by
\begin{equation*}
b_{h}^{E}(q_{h},v_{h}) = b^{E}(\Pio_k q_{h}, \Pio_k v_{h}) + S_{0}^{E}(q_{h} - \Pio_k q_{h}, v_{h} - \Pio_k v_{h})\quad \forall q_{h}, v_{h} \in V_{h}^{E}.
\end{equation*}
We remark that the discrete bilinear form $b_{h}^{E}(\cdot,\cdot)$ satisfies the classical properties of consistency and stability (which are similar to those satisfied by $a_h^{E}(\cdot,\cdot)$). In particular, there exists two positive constants $\beta_{*}, \beta^*$ independent of $h$ and E, such that 
\begin{equation*}
\beta_*b^{\E}(q_h,q_h)\leq b_h^{\E}(q_h,q_h)\leq\beta^*b^{\E}(q_h,q_h)\quad\forall q_h\in\VK.
\end{equation*} 
Then, the global discrete bilinear forms $a_{h}(\cdot,\cdot)$ and $b_{h}(\cdot,\cdot)$ are be expressed componentwise as follows
 \begin{equation*}
\label{eq:bilineal_form_B_split_{h}}
\begin{split}
a_{h}(q_{h},v_{h}):
= \sum_{ E\in\CT_h} a_{h}^{E}( q_{h}, v_{h}), \quad
b_{h}(q_h,v_h) := \sum _{E\in\CT_h} b_{h}^{E}(q_h,v_h), \\
\widehat{a}_{h}(q_{h},v_{h}) := \sum_{ E\in\CT_h} a_{h}^{E}( q_{h}, v_{h}) + b_{h}^{E}(q_h,v_h).
\end{split}
\end{equation*}

\begin{remark}
Let us remark that the definition of $b_h(\cdot,\cdot)$ is needed in order to obtain a $V_h$-coercive bilinear form on the left hand side, namely, $\widehat{a}_h(\cdot,\cdot)$. However, for the  numerical experiments is not necessary to define $b_h(\cdot,\cdot)$ with the stabilization term $S_0^E(\cdot,\cdot)$,  because at computational level,  there is no need to solve the problem with the shifted formulation.\end{remark}

Now, according to \cite{MR3507277, GMV2018}, the nature of the NCVEM demands the introduction of the conformity error term, denoted by $\mathcal{N}_h(\cdot,\cdot)$ and defined, for the solution  $\widehat{p}\in\H^1(\O)$
of \eqref{eq:source_pr} and $v_h\in\H^{1,nc}(\CT_h)$, by
\begin{equation*}
\mathcal{N}_h(\widehat{p},v_h):=\widehat{a}(\widehat{p},v_h)-b(f,v_h)=\sum_{\ell\in\mathcal{E}_h}\frac{c^2}{\rho}\int_{\ell}\nabla \widehat{p}\cdot\jump{v_h}.
\end{equation*}
The conformity error term satisfies the following estimate (see \cite[Lemma 3.2]{GMV2018}).
\begin{lemma}
\label{lmm:bound_Nh}
Under assumptions {\bf A1} and {\bf A2} , let $\widehat{p}\in\H^{1+r}(\O)$ with $s>1/2$ be the solution of \eqref{eq:source_pr} and let  $v_h\in\H^{1,nc}(\CT_h)$. Then, there exists a constant $C>0$ depending on $c^2$, $\rho$, and the mesh regularity such that 
\begin{equation*}
|\mathcal{N}_h(\widehat{p},v_h)|\leq Ch^t|\widehat{p}|_{1+s,\O}|v_h|_{1,h},
\end{equation*}
where $t:=\min\{s,k\}$.
\end{lemma}
\subsection{Spectral discrete problem}
Now we introduce the VEM discretization of problem \eqref{eq:pression}. To do this task, we require the global space $V_h$
defined in \eqref{eq:globa_space} together with the assumptions introduced in Section \ref{sec:virtual}.

Setting $\lambda_{h} := \omega_{h}^{2}+1$, the discrete spectral problem reads as follows: Find $\lambda_h\in\mathbb{R}$ and $0\neq p_h\in V_h$ such that
\begin{equation}
\label{eq:spectral_disc}
\widehat{a}_{h}(p_h,v_h)=\lambda_{h} b_h(p_h,v_h) \quad \forall v_h \in V_h.
\end{equation}

It is possible to prove that $\widehat{a}_h(\cdot,\cdot)$ is $V_h$-coercive. Indeed, for $v_h\in V_h$, we have
\begin{multline*}
\label{eq:ah_coercive}
\widehat{a}_h(v_h,v_h)=\sum_{E\in\CT_h}a_h^E(v_h,v_h)+b_h^E(v_h,v_h)\\
 \geq \underbrace{\frac{C}{\rho}\min\{ c^2,1\}\min\{\alpha_*,\beta_*\}}_{\widetilde{C}}\sum_{E\in \CT_h}\|v_h\|_{1,E}^2=\widetilde{C}\|v_h\|_{1,h}^2.
\end{multline*}
Let  us introduce the discrete solution operator $T_h : \L^2(\O) \rightarrow \Vh$, defined by $T_h f_h := \widehat{p}_h$, where $\widehat{p}_h$ is the unique solution of the corresponding associated source problem: Find $\widehat{p}_h \in \Vh$ such that 
\begin{equation*}
\widehat{a}_h(\widehat{p}_h,v_h) := b_h(f_h,v_h) \quad \forall v_h \in \Vh,
\end{equation*}
which according to Lax-Milgram's lemma is well defined.

No we introduce some necessary technical results that are useful to analyze convergence and  error estimates of the method.
\begin{lemma}[Existence of a virtual approximation operator]
\label{lmm:polyaprox}
If assumption {\bf A1} is satisfied, then  for every $s$ with 
$0\le s\le k$ and for every $q\in\H^{1+s}(E)$, there exists
$q_{\pi}\in\mathbb{P}_k(\E)$ such that
$$
\left\|q-q_{\pi}\right\|_{0,\E}
+h_{\E}\left|q-q_{\pi}\right|_{1,\E}
\leq C h_{\E}^{1+s}\left\|q\right\|_{1+s,\E},
$$
where the   constant  $C>0$ depends only on mesh regularity constant  $\gamma$.
\end{lemma}
Finally, we have the following result that provides the existence of an interpolant operator on the virtual space (see \cite[Proposition 4.2]{MR3340705}).
\begin{lemma}[Existence of an interpolation operator]\label{lmm:interpolate}
\label{eq:interpolant}
Under the assumptions {\bf A1} and  {\bf A2} , let $q \in \H^{1+s}(\O)$, with $0\leq s\leq k$. Then, there exists $q_I\in \Vh$ such that
\begin{equation*}
\|q-q_I\|_{0,E}+h_E|q-q_I|_{1,E}\leq C h_E^{1+s}|q|_{1+s,E},
\end{equation*}
where $C$ is  positive and independent of $h_E$.
\end{lemma}

Our next task is to check  the following properties of the non-compact operators theory \cite{DNR1}:
\begin{itemize}
\item P1: $\|T-T_h\|_h:=\displaystyle\sup_{0\neq f_h\in V_h}\frac{\|(T-T_h)f_h\|_{1,h}}{\|f\|_{1,h}}\rightarrow 0$ as $h\rightarrow 0$;
\item P2: $\forall q\in V_h$, $\lim_{h\rightarrow 0}\delta(q, V_h)=0$.
\end{itemize}

Since P2 is immediate due  the density of linear polynomials on $\L^2$ and the approximation property in Lemma \ref{lmm:interpolate}. Hence, we only prove property P1. We begin with 
the following approximation result.
\begin{lemma}
\label{lmm:aproxTTh1}
Let $f_h\in V_h$ such that $Tf_h:=\widehat{p}$ and $T_hf_h:=\widehat{p}_h$  Then, there exists a constant $C>0$ such that 
\begin{equation*}
\|(T-T_h)f_h\|_{1,h}\leq C^*h^{\min\{s,1\}}\|f_h\|_{1,h}.
\end{equation*}
\end{lemma}
\begin{proof}
Let $f_h\in V_h$ and let $\widehat{p}_I\in V_h$ be the virtual  interpolant of $\widehat{p}$. Then, from the definition of the norm we have
\begin{equation*}
\|\widehat{p}-\widehat{p}_h\|_{1,h}\leq\|\widehat{p}-\widehat{p}_I\|_{1,h}+\|\widehat{p}_I-\widehat{p}_h\|_{1,h}.
\end{equation*}
The first term on the right hand side of the estimate above is bounded as follows
\begin{multline}
\label{eq:fh1}
\|\widehat{p}-\widehat{p}_I\|_{1,h}=\sum_{E\in\CT_h}\|\widehat{p}-\widehat{p}_I\|_{1,E}\leq C\sum_{E\in\CT_h} h_E^{s}|\widehat{p}|_{1+s,E}\\
\leq Ch^s|\widehat{p}|_{1+s,\O}\leq Ch^{s}\|f_h\|_{1,h}.
\end{multline}
Now, for the remaining term, setting $v_h := \widehat{p}_I - \widehat{p}_h$ and invoking the coercivity of $\widehat{a}_h(\cdot,\cdot)$, we obtain the following identity
\begin{multline*}
\widetilde{C}\|v_h\|_{1,h}^2\leq \widehat{a}_h(\widehat{p}_I-\widehat{p}_h,v_h)
=\widehat{a}_h(\widehat{p}_I,v_h)-\widehat{a}_h(\widehat{p}_h,v_h) = \widehat{a}_h(\widehat{p}_I,v_h)-b_h(f_h,v_h)\\
=\sum_{E\in\CT_h}\{\widehat{a}_h^E(\widehat{p}_I-\widehat{p}_\pi,v_h)+\widehat{a}^E(\widehat{p}_\pi-\widehat{p},v_h) +
\widehat{a}^E(\widehat{p},v_h) \}-b_h(f_h,v_h)\\
=\sum_{E\in\CT_h}\{\widehat{a}_h^E(\widehat{p}_I-\widehat{p}_\pi,v_h)+\widehat{a}^E(\widehat{p}_\pi-\widehat{p},v_h)\}+b(f_h,v_h)-b_h(f_h,v_h)+\mathcal{N}_h(\widehat{p},v_h)\\
\leq \underbrace{\dfrac{\max\{c^{2},1\}\max\{c_1,b_1\}}{\rho}}_{\widetilde{C}}\left(\sum_{E\in\CT_h}\|\widehat{p}_I-\widehat{p}_\pi\|_{1,E}\|v_h\|_{1,E}+\sum_{E\in\CT_h}\|\widehat{p}_I-\widehat{p}\|_{1,E}\|v_h\|_{1,E}\right)\\
+h\|f_h\|_{0,\O}\|v_h\|_{1,h}+\mathcal{N}_h(\widehat{p},v_h)
\\
\leq \widetilde{C}\left(\sum_{E\in\CT_h}(2\|\widehat{p}-\widehat{p}_I\|_{1,E}+\|\widehat{p}-\widehat{p}_\pi\|_{1,E})\|v_h\|_{1,E}\right)+h\|f_h\|_{0,\O}\|v_h\|_{1,h}+\mathcal{N}_h(\widehat{p},v_h)\\
\leq 3\widetilde{C}h^s|\widehat{p}|_{1+s,\O}\|v_h\|_{1,h}+h\|f_h\|_{0,\O}\|v_h\|_{1,h}\leq \widehat{C} h^{\min\{s,1\}}\|f_h\|_{0,\O}\|v_h\|_{1,h},
\end{multline*}
where $\widehat{C}:=3\widetilde{C}$. The previous estimate implies that
\begin{equation}
\label{eq:fh2}
\|\widehat{p}_I-\widehat{p}_h\|_{1,h}\leq \widehat{C}h^{\min\{s,1\}}\|f_h\|_{1,h}.
\end{equation}
Hence, gathering \eqref{eq:fh1} and \eqref{eq:fh2} we conclude the proof, with $C^* := C\max\{\widehat{C},1\}$.
\end{proof}
 Now we are in position to establish property  P1.
 \begin{corollary}
 \label{lmm:P1}
 Property P1 holds true.
 \end{corollary}
 \begin{proof}
Follows directly from Lemma \ref{lmm:aproxTTh1}.
 \end{proof}
 \begin{remark}
 Let $\widehat{p}\in \H^{1+s}(\O)$ with $s$ as in Lemma \ref{lmm:regularity} such that  $Tf=\widehat{p}$ and $T_hf=\widehat{p}_h$, we obtain
 \begin{equation*}
\|(T-T_h)f\|_{1,h}\leq C\left(h^{\min\{s,k\}}|\widehat{p}|_{1+s,\O}+h	\sum_{E\in\mathcal{T}_h}(\|f-\Pi_k^E f\|_{0,E})\right).
\end{equation*}
Then, according to \cite{MR3507277,GMV2018}. if $f\in \H^{1+s}(\O)$ we have
\begin{equation}
\label{eq:estimate_dnr_reg}
\|(T-T_h)f\|_{1,h}\leq C\left(h^{\min\{s,k\}}\|f\|_{0,\O})\right).
\end{equation}
Let us remark that, since we have using the theory of non-compact operators, when we consider the source  $f \in \L^2(\O)$ the estimates hold for the lowest order case, i.e $k=1$. However, assuming more regularity for $f$ as in the previous estimate, we care able to consider the standard estimates, where the powers of $h$ are as in \eqref{eq:estimate_dnr_reg}.
 \end{remark}

 \subsection{Error estimates}
We begin this section recalling some definitions of spectral theory. Let $\mathcal{X}$ be a generic Hilbert space and let $S$ be a linear bounded operator defined by $S:\mathcal{X}\rightarrow\mathcal{X}$. If $I$ represents the identity operator, the spectrum of $S$ is defined by $\sp(S):=\{z\in\mathbb{C}:\,\,(z I-S)\,\,\text{is not invertible} \}$ and the resolvent is its complement $\varrho(S):=\mathbb{C}\setminus\sp(S)$. For any $z\in\varrho(S)$, we define the resolvent operator of $S$ corresponding to $z$ by $R_z(S):=(z I-S)^{-1}:\mathcal{X}\rightarrow\mathcal{X}$. 
Also, if $\mathcal{X}$ and $\mathcal{Y}$ are vectorial fields, we denote by $\mathcal{L}(\mathcal{X},\mathcal{Y})$ the space of all the linear and bounded operators acting from $\mathcal{X}$ to $\mathcal{Y}$.

Now our task is to obtain error estimates for the approximation of the eigenvalues and eigenfunctions. With this goal in mind, first  we need to recall the definition of the  gap $\widehat{\delta}$ between two closed subspaces $\CM$ and $\CN$ of  $\H^1(\Omega)$:
\begin{equation*}
\widehat{\delta}(\CM,\CN) := \max\{\delta (\CM, \CN), \delta (\CM, \CN)\},
\end{equation*}
where $\delta (\CM,\CN):= \underset{x \in \CM : \|x\|_{1, \Omega} = 1}{\sup} \left\lbrace \underset{y \in \CN}{\inf} \quad \|x-y\|_{1, \Omega} \right\rbrace$.

Let us  recall the definitions of the resolvent operators of $T$ and $T_{h}$ respectively:
\begin{gather*}
	R_z(T):=(z I-T)^{-1}\,:\, \H^{1,nc}(\CT_h,k) \to \H^{1,nc}(\CT_h,k)\,, \quad z\in\mathbb{C}\setminus \sp(T), \\
	R_z(T_{h}):=(z I-T_{h})^{-1}\,:\, V_h \to V_h\,, \quad z\in\mathbb{C}\setminus\sp(T_{h}) .
\end{gather*}

Also, for the resolvent of the discrete operator, we have the following result.
  \begin{lemma}
 \label{thm:bounded_resolvent}
 Let $F\subset\varrho(T)$ be closed. Then, there exist
 positive constants $C$ and $h_0$, independent of $h$, such that for $h<h_0$
\begin{equation*}
 \displaystyle\|(zI-T_{h})^{-1}f_h\|_{1,h}\leq C\|f_h\|_{1,h}\qquad\forall z\in F.
 \end{equation*}
 \end{lemma}

Now we prove the following result for the resolvent of $T$ on the broken norm.
\begin{lemma}
\label{lmm:borekn_resolvent}
Let $z\in F\subset\varrho(T)$. Then, there exists a constant $C>0$ independent of $h$ such that 
\begin{equation*}
\|(z I-T)v\|_{1,h}\geq C|z|\|v\|_{1,h}\quad\forall v\in \H^{1,nc}(\CT_h,k).
\end{equation*}
\end{lemma}
\begin{proof}
Let us set $p^*:=T v\in \H^1(\O)$ and consider the following identity
\begin{equation*}
(z I-T)p^*=T(z I-T)v.
\end{equation*}
Let us remark that $T$ is a bounded operator and hence, $\|(z I-T)p\|_{1,\O}\geq C\|p\|_{1,\O}$ holds. Then, following the proof of \cite[Lemma 4.1]{MR4077220} we have 
\begin{equation*}
 C\|p^*\|_{1,\O} \leq \|(z I-T)p^*\|_{1,\O}\leq \|T\|_{\mathcal{L}(\L^2(\O),\H^1(\O))}\|(z I-T)v\|_{1,h},
\end{equation*}
whereas,
\begin{equation*}
C\|v\|_{1,h}\leq |z|^{-1}\|p^*\|_{1,\O}+|z|^{-1}\|(z I-T)v\|_{1,h}\leq |z|^{-1} C\|(z I-T)v\|_{1,h},
\end{equation*}
implying that $C|z|\|v\|_{1,h}\leq \|(z I-T)v\|_{1,h}$. This concludes the proof.
\end{proof}
As a consequence of the above lemma, we have that the resolvent of $T$ is bounded. This is equivalent to say that there exista a constant $C>0$ such that for any $z\in F\subset\varrho(T)$ there holds
\begin{equation}\label{eqr11}
\| (z I-T)^{-1}\|_{h}\leq C\quad\forall z\in F.
\end{equation}

Now our aim is to prove that for $h$ small enough, the discrete resolvent is also bounded. To do this task, we simply resort to \cite[Lemma 4.2]{MR4077220}.
\begin{lemma}
\label{lmm_disc_res_bound}
If $z\in F\subset\varrho (T)$, there exists $h_0>0$ such that for all $h\leq h_0$, there exists a constant $C>0$ independent of $h$ but depending on $|z|$ such that
\begin{equation*}
\|(z I-T_h)v_h\|_{1,h}\geq C\|v_h\|_{1,h}\quad\forall v_h\in V_h.
\end{equation*}
\end{lemma}

The previous lemma states that if we consider a compact subset $F$ of the complex plane such that
$F\cap\varrho(T)=\emptyset$ for $h$ small enough and for all $z\in F$, the operator $z I-T_h$ is invertible and there exists a constant $C>0$ independent of $f$ such that
\begin{equation*}
\| (z I-T_h)^{-1}\|_{h}\leq C\quad\forall z\in F.
\end{equation*}

We present as a consequence of the above, that the proposed method does not introduce spurious eigenvalues. In practical terms, this implies that isolated parts of $\sp(T)$ are approximated by isolated parts of $\sp(T_h)$ (see \cite{MR0203473}). This is contained in the following result.

\begin{theorem}
Let $G \subset \mathbb{C}$  be an open set containing $\sp(T)$. Then, there exists $h_0>0$ such that $\sp(T_h)\subset G$ for all $h<h_0$.
\end{theorem}

Now, let $\mu$ be an isolated eigenvalue if $T$ with multiplicity $m$ and let $\mathcal{E}$ be its associated eigenspace. Then, there exist $m$ eigenvalues $\mu_h^{(1)},\ldots, \mu_h^{(m)}$ of $T_h$, repeated according to their respective multiplicities that converge to $\mu$. Now, let $\mathcal{E}_h$ be the  direct sum of the associated eigenspaces of $\mu_h^{(1)},\ldots, \mu_h^{(m)}$. With these definitions at hand, now we focus on the analysis of error estimates.

Our next task is to derive error estimates for the eigenvalues and eigenfunctions. 
Let $\mu \in (0,1)$ be an isolated eigenvalue of $T$ and let $\textrm{D}$ be an open disk in the complex plane with boundary $\gamma$, such that $\mu$ is the only eigenvalue of $T$ that lying on $\textrm{D}$ and $\gamma \cap \sp(T) = \emptyset$. Let $\mathfrak{E}:\H^{1,nc}(\CT_h,k)\rightarrow \H^{1,nc}(\CT_h,k)$ be the spectral projector of $T$ corresponding to the isolated 
eigenvalue $\xi$, namely
\begin{equation*}
\displaystyle \mathfrak{E}:=\frac{1}{2\pi i}\int_{\gamma} R_{z}(T)dz.
\end{equation*}
On the other, we define $\mathfrak{E}_h:V_h\rightarrow\Vh$ as the spectral projector of $T_{h}$ corresponding to the isolated 
eigenvalue $\xi_h$, namely
\begin{equation*}
\displaystyle \mathfrak{E}_h:=\frac{1}{2\pi i}\int_{\gamma} R_{z}(T_{h})dz.
\end{equation*}

where $\mathfrak{E}$ is well defined and bounded uniformly in $h$ due to \eqref{eqr11}. On the other hand, observe that $\mathfrak{E}|_{\H^{1}(\O)}$ is a spectral projection in $\H^{1}(\O)$ onto the eigenspace $\mathfrak{E}(\H^{1,nc}(\CT_h,k))$ corresponding to the eigenvalue $\mu$ of $T$. Moreover, we have
\begin{equation*}
\mathfrak{E}(\H^{1,nc}(\CT_h,k)) = \mathfrak{E}(\H^{1}(\O)).
\end{equation*}
Also, $\mathfrak{E}_h|_{\Vh}$ is a projector in $\Vh$ onto the eigenspace $\mathfrak{E}_h(\Vh)$ corresponding to the eigenvalues o $T_h$ contained in $\gamma$. As in the continuous case, we have
\begin{equation*}
\mathfrak{E}_h(\H^{1,nc}(\CT_h,k)) = \mathfrak{E}_h(\Vh).
\end{equation*}

We can prove the following result (see \cite[Lemma 4.4]{MR4077220}.
\begin{lemma}
\label{lemma:spectral_projectors}
Let $f_h\in \Vh$. There exist constants $C>0$ and $h_{0}>0$ such that, for all $h<h_{0}$,
\begin{equation*}
	\|  (\mathfrak{E}-\mathfrak{E}_{h})f_h \|_{1,h}\leq C\|(T-T_{h})f_h\|_{1,h}\leq C\,h^{\min\{s,1\}} \| f_h \|_{1,h}.
\end{equation*}
\end{lemma}

Now, we present the following result, that will be used to establish approximation properties between the discrete and continuous eigenfunctions. For details about these following results, we resort  to \cite{MR4077220}.

\begin{lemma}
There exists $C>0$ independent of $h$ such that 
\begin{equation*}
\widehat{\delta}(\mathfrak{E}(\H^{1}(\O)),\mathfrak{E}_h(\Vh)) \leq C(\|T-T_h\|_{h} + \delta(\mathfrak{E}(\H^{1}(\O)),\Vh)).
\end{equation*}
\end{lemma}
Now, for a fixed eigenvalue $\mu \in (0,1)$ of $T$, we have the following result.
\begin{theorem}
For $h$ small enough, there exists $C>0$ independent of $h$, such that
\begin{equation*}
\widehat{\delta}(\mathfrak{E}(\H^{1}(\O)),\mathfrak{E}_h(\Vh)) \leq C\delta(\mathfrak{E}(\H^1(\O)),\Vh).
\end{equation*}
\end{theorem}

The following result provides an error estimate for the eigenfunctions and eigenvalues.
\begin{theorem}
\label{thm:errors1} 
The following estimates hold
\begin{equation*}
\widehat{\delta}(\mathcal{E}, \mathcal{E}_{h}) \leq \widetilde{C}h^{\min\{r,k\}}, \quad |\mu - \mu_{h}^{(i)}| \leq \widetilde{C}h^{\min\{r,k\}}, \quad i = 1, \ldots, m,
\end{equation*}
where $r$ is given in Lemma \ref{lmm:regularity} and $\widetilde{C}>0$ is independent of $h$.
\end{theorem}

Let us observe that Theorem \ref{thm:errors1} is a result with a preliminary error estimate for the eigenvalues. However, it is possible to improve the order of convergence for the eigenvalues as we prove on the following result.
\begin{theorem}
The following estimate holds
	\begin{equation*}
	\label{eq:double_order}
		|\lambda-\lambda_{h,i}|\leq \mathcal{K} h^{2\min\{r,k\}},
	\end{equation*}
where $\mathcal{K}>0$ is a constant independent of $h$ and  $r$ is given in Lemma \ref{lmm:regularity}.
\end{theorem}
\begin{proof}
Let $(\l_{h.i},p_h)\in\mathbb{R}\times V_h$ be   a solution of \eqref{eq:spectral_disc} with $\left\|p_h\right\|_{1,h}=1$. According to Lemma \ref{lmm:aproxTTh1}, there exists a solution $(\l,p)\in\mathbb{R}\times \H^1(\O)$ of the eigenvalue problem \eqref{eq:pression} such that $\left\|p-p_h\right\|_{1,h}\leq C^*h^{\min\{r,k\}}$. 

From the symmetry of the bilinear forms and the facts that $\widehat{a}(p,v)=\l b(p,v)$ for all $v\in\H ^1(\O)$ (cf.\eqref{eq:pression}) and $\widehat{a}_h(p_h,v_h)=\l_{h,i}b(p_h,v_h)$ for all $v_h\in\Vh$ (cf.\eqref{eq:spectral_disc}), we have
\begin{multline*}
\widehat{a}(p-p_h,p-p_h)-\l b(p-p_h,p-p_h)
 =\left[\widehat{a}(p_h,p_h)-\widehat{a}_h(p_h,p_h)\right]
-\left(\l-\l_{h,i}\right)b(p_h,p_h) \\ 
+ 2\mathcal{N}_h(p,p-p_h) + \lambda_{h,i}\left[b_h(p_h,p_h) - b(p_h,p_h)\right],
\end{multline*}
where the following identity can  be obtained
\begin{multline*}
\left(\lambda-\lambda_{h}^{(i)}\right)b(p_h,p_h) = \underbrace{\l b(p-p_h,p-p_h) - \widehat{a}(p-p_h,p-p_h)}_{\textrm{(I)}} + \underbrace{2\mathcal{N}_h(p,p-p_h)}_{\textrm{(II)}} \\
+ \underbrace{\left[\widehat{a}(p_h,p_h)-\widehat{a}_h(p_h,p_h)\right]}_{\textrm{(III)}} + \underbrace{\lambda_{h,i)}\left[b_h(p_h,p_h) - b(p_h,p_h)\right]}_{\textrm{(IV)}}.
\end{multline*}
Our task is to estimate the contributions on the right hand side in the above estimate. For \textrm{(I)}, invoking Theorem \ref{thm:errors1} we have
\begin{equation}
\label{eq:boundI}
|\textrm{(I)}| \leq \dfrac{\max\{c^{2},1\}}{\rho}\|p-p_h\|_{1,h}^{2} \leq \underbrace{C\dfrac{\max\{c^{2},1\}}{\rho}}_{C_1}h^{2\min\{r,k\}}.
\end{equation}
To estimate \textrm{(II)}, invoking Lemma \ref{lmm:bound_Nh} and Theorem \ref{thm:errors1}, we obtain
\begin{equation}
\label{eq:boundII}
|\textrm{(II)}| \leq Ch^{\min\{r,k\}}|p-p_h|_{1,h} \leq \underbrace{C\widetilde{C}}_{C_2}h^{2\min\{r,k\}}.
\end{equation}
Now, for \textrm{(III)} we have by the consistency of $\widehat{a}_{h}^{E}(\cdot,\cdot)$ and Lemmas \ref{lmm:polyaprox} and \ref{thm:errors1} that 
\begin{multline}
\label{eq:boundIII}
|\textrm{(III)}| = \left\lvert\sum_{E\in\CT_h} \widehat{a}_{h}^{E}(p_h-p_\pi,p_h-p_\pi) - \widehat{a}^{E}(p_h-p_\pi,p_h-p_\pi)\right\lvert \\
\leq \underbrace{C\max\{\alpha^*,1\}\dfrac{\max\{c^{2},1\}}{\rho}}_{C_3}\|p_h-p_\pi\|_{1,h}^{2} \leq C_2(\|p-p_h\|_{1,h}+\|p-p_\pi\|_{1,h})^{2} \\
\leq \underbrace{C_{3}\max\{\widetilde{C},1\}}_{C_4}h^{2\min\{r,k\}}.
\end{multline}
Finally, for \textrm{(IV)} using approximation properties of $\Pi_k$ and it stability, together with triangle inequality we obtain
\begin{equation}
\label{eq:boundIV}
\textrm{(IV)} \leq C\|p_h-\Pi_kp_h\|_{0,\O}^{2} \leq 2C\|p-p_h\|_{0,\O}^{2}+C\|p-\Pi_k p\|_{0,\O}^{2} \leq 2C\widetilde{C}h^{2\min\{r,k\}}.
\end{equation}
Then, due to the coercivity of $\widehat{a}_{h}(\cdot,\cdot)$ (cf. \eqref{eq:ah_coercive}) and the fact that $\lambda_{h}^{(i)} \rightarrow \lambda$ as $h\rightarrow 0$ we obtain
\begin{equation}
\label{eq:bound_b}
b_{h}(p_h,p_h) = \dfrac{\widehat{a}_h(p_h,p_h)}{\lambda_{h,i}} \geq \dfrac{\widetilde{C}}{\lambda_{h}^{(i)}}\|p_h\|_{1,h}^{2} \geq C > 0.
\end{equation}
Hence, gathering \eqref{eq:boundI}, \eqref{eq:boundII}, \eqref{eq:boundIII}, \eqref{eq:boundIV}, and \eqref{eq:bound_b} we conclude the proof, where 
\begin{equation*}
\mathcal{K}:=\max\{C_{1},C_{2},C_{4},C\widetilde{C}\}.
\end{equation*}
\end{proof}
\subsection{Error estimates in $\L^2$ norm}
In the present subsection we establish error estimates for eigenfunctions in $\L^{2}$ norm. We begin this subsection with the following result, where a classical duality argument has been used.
\begin{lemma}\label{eq:cvlm2}
Let $f_{h} \in V_h$ be such that $\widehat{p} := Tf_{h}$ and $\widehat{p}_{h} := T_{h}f_{h}$. Then, the following estimate  holds
\begin{equation*}
\|\widehat{p}-\widehat{p}_{h}\|_{0,\O} \leq \mathcal{J}h^{\sigma}\|f\|_{1,h}, \quad \sigma := \min\{\widetilde{s},1\}+\min\{s,1\},
\end{equation*}
where $\mathcal{J}$ is a positive constant independent of $h$, $\widetilde{s}>0$ is the regularity associated with an auxiliary problem and  $s$ is given by Lemma \ref{lmm:regularity}.
\end{lemma}

\begin{proof}
Let us consider the following auxiliarly problem: Find $q \in \H_0^1(\O)$ such that 
\begin{equation}\label{eq:auxprob}
-\dfrac{c^{2}}{\rho}\Delta q = \dfrac{1}{\rho}(\widehat{p}-\widehat{p}_h) \quad \text{in} \; \O, \qquad q = 0 \quad \text{on} \; \partial \O.
\end{equation}
Observe that \eqref{eq:auxprob} is well defined and it solution $q \in \H_0^1(\O)$ satisfies the following estimate 
\begin{equation*}\label{eq:auxreg}
|q|_{1+\widetilde{s},\O} \leq C\|\widehat{p}-\widehat{p}_h\|_{0,\O}.
\end{equation*}
Now, testing \eqref{eq:auxprob} with $v := \widehat{p}-\widehat{p}_h \in \H^{1,nc}(\CT_h)$ we have
\begin{equation}\label{eqerr}
\begin{split}
\dfrac{1}{\rho}\|\widehat{p}-\widehat{p}_h\|_{0,\O}^{2} &= -\dfrac{c^{2}}{\rho}\int_\O \Delta q (\widehat{p}-\widehat{p}_h) \\
&= \dfrac{c^{2}}{\rho}\sum_{E\in\CT_h} \int_{E} \nabla q \cdot \nabla (\widehat{p}-\widehat{p}_h) + \dfrac{c^{2}}{\rho}\sum_{E\in\CT_h} \int_{\partial E} (\nabla q \cdot \bn)(\widehat{p}-\widehat{p}_h) \\
&= a(q,\widehat{p}-\widehat{p}_h) + \dfrac{c^{2}}{\rho}\mathcal{N}_{h}(q,\widehat{p}-\widehat{p}_h) \\
&= a(q-q_{I},\widehat{p}-\widehat{p}_h) + \dfrac{c^{2}}{\rho}\mathcal{N}_h(q,\widehat{p}-\widehat{p}_h) + a(q_{I},\widehat{p}-\widehat{p}_h).
\end{split}
\end{equation}
Observe that the first two terms in the previous identity can be estimate easily. Indeed, invoking Lemmas \ref{eq:interpolant} and \ref{lmm:aproxTTh1} we have
\begin{equation}\label{eq11}
|a(q-q_{I},\widehat{p}-\widehat{p}_h)| \leq \dfrac{c^{2}}{\rho}|q-q_I|_{1,h}|\widehat{p}-\widehat{p}_h|_{1,h} \leq \underbrace{C^*\dfrac{c^{2}}{\rho}}_{R_1}h^{\sigma}|q|_{1+\widetilde{s},\O}\|f_{h}\|_{1,h}.
\end{equation}
On the other hand, invoking Lemmas \ref{lmm:bound_Nh} and \ref{eq:interpolant} we obtain
\begin{equation}\label{eq22}
\dfrac{c^{2}}{\rho}\mathcal{N}_{h}(q,\widehat{p}-\widehat{p}_h) \leq \underbrace{C\dfrac{c^{2}}{\rho}}_{R_2}h^{\sigma}|q|_{1+\widetilde{s},\O}\|f_{h}\|_{1,h}.
\end{equation}
Finally, we have the following identity
\begin{equation*}
\begin{split}
a(q_{I},\widehat{p}-\widehat{p}_h) &= a(q_{I},\widehat{p}) - a(q_{I},\widehat{p}_h) \\
&= \mathcal{N}_{h}(\widehat{p},q_{I}) + b(f_h,q_I) - a(\widehat{p}_h,q_I) \\
&= \underbrace{\mathcal{N}_h(\widehat{p},q_I)}_{\textrm{T}_1} + \underbrace{\left[a_h(\widehat{p}_h,q_I) - a(\widehat{p},q_I-q)\right]}_{\textrm{T}_2} + \underbrace{\left[b(f_{h},q_I) - b_h(f_{h},q_I)\right]}_{\textrm{T}_3}.
\end{split}
\end{equation*}
Where now the task is to estimate each of the contributions on the right hand side of the above identity. To estimate $\textrm{T}_1$ we invoke Lemmas \ref{lmm:bound_Nh} and \ref{eq:interpolant} in order to obtain
\begin{equation*}
|\textrm{T}_1| = |\mathcal{N}_h(\widehat{p},q_I-q)| \leq Ch^{\sigma}\|f_{h}\|_{1,h}|q|_{1+\widetilde{s},\O}.
\end{equation*}
For $\textrm{T}_2$, we use the stability of $a_h(\cdot,\cdot)$ and triangle inequality, obtaining
\begin{equation*}
\begin{split}
|\textrm{T}_2| &= \left\lvert \sum_{E\in\CT_h} a_h^E(\widehat{p}_h,q_I) - a^E(\widehat{p}_h,q_I)\right\lvert \\
&= \left\lvert \sum_{E\in\CT_h} a_h^E(\widehat{p}_h-\widehat{p}_\pi,q_I-q_\pi) - a^E(\widehat{p}_h-\widehat{p}_\pi,q_I-q_\pi)\right\lvert \\
&\leq \max\{\alpha^*,1\}\dfrac{c^{2}}{\rho}|\widehat{p}_h-\widehat{p}_\pi|_{1,h}|q_I-q_\pi|_{1,h},
\end{split}
\end{equation*}
where in the last inequality we have used triangle inequality and the Cauchy-Schwarz. Hence, using again triangle inequality and invoking Lemmas \ref{lmm:aproxTTh1}, \ref{lmm:polyaprox} and \ref{eq:interpolant} we obtain
\begin{equation*}
|\textrm{T}_2| \leq C^*\max\{\alpha^*,1\}\dfrac{c^{2}}{\rho}h^{\sigma}\|f_{h}\|_{1,h}|q|_{1+\widetilde{s},\O}. 
\end{equation*}

Finally, for $\textrm{T}_3$, we use the stability of $b_h(\cdot,\cdot)$, the  Cauchy-Schwarz inequality,  and approximation properties of $\Pi_k$, obtaining
\begin{equation*}
|\textrm{T}_3| \leq C\|f_{h}-\Pi_k f_{h}\|_{0,\O}\|q_{I}-\Pi_k q_I\|_{0,\O} \leq Ch^{\sigma}\|f_{h}\|_{1,h}|q|_{1+\widetilde{s},\O}.
\end{equation*}
From this, we conclude that 
\begin{equation}\label{eq33}
|a(q_I,\widehat{p}-\widehat{p}_h)| \leq \underbrace{C\max\{1,C^*\max\{\alpha^*,1\}c^{2}\rho^{-1}\}}_{R_3}h^{\sigma}\|f_{h}\|_{1,h}|q|_{1+\widetilde{s},\O}.
\end{equation}
Therefore, we obtain
\begin{equation*}
\dfrac{1}{\rho}\|\widehat{p}-\widehat{p}_h\|_{0,\O}^{2} \leq \max\{R_1,R_2,R_3\}h^{\sigma}\|f_{h}\|_{1,h}|q|_{1+\widetilde{s},\O}.
\end{equation*}

Finally, gathering \eqref{eq11}, \eqref{eq22} and \eqref{eq33} and replacing it in \eqref{eqerr}, together with the additional regularity for $q$, allow to us to conclude that 
\begin{equation*}
\|\widehat{p}-\widehat{p}_h\|_{0,\O} \leq \mathcal{J}h^{\sigma}\|f_{h}\|_{1,h}, \quad \mathcal{J} := \rho\max\{R_1,R_2,R_3\}.
\end{equation*}
\end{proof}

\section{Numerical experiments}
\label{sec:numerics}
The following section contains a series of numerical tests where the aim is to corroborate the theoretical results obtained in the previous sections. Here, our goals are two main subjects: in one hand, to analyze the accuracy of the method on the approximation of the eigenvalues, where in different two dimensional domains and different polygonal meshes, we run a Matlab code  and compute converge order of approximation, which according to our theory must be two. This is of course performed for the lowest order NCVEM. On the other hand, we will confirm that our method is spurious free. This analysis is required since in its nature, the VEM, on its conforming and non-conforming versions, depends on some stabilizations that may  introduce spurious eigenvalues if it is not correctly chosen. We will check that always is possible to find a threshold in which the method captures safely the physical spectrum.

Depending on the chosen geometry, the eigenvalues and eigenfunctions are exactly known, so our results will be compared with the exact ones. In  the case where the geometry does not allow us to obtain exact eigenvalues, we resort to a least-square fitting to obtain order of convergence and extrapolated values that we compare with the literature. 

To complete the choice of the NCVEM, we had to fix the bilinear form $S^E(\cdot, \cdot)$ satisfying \eqref{eq:stability} to be used. To do this, we proceeded as in \cite{MR3507277,GMV2018}.
\begin{equation*}
S^E(u_h,v_h)=\sigma_E\bu_h\bv_h\quad\text{and}\quad
S_0^E(u_h,v_h)=\tau_Eh_E^2\bu_h\bv_h,
\end{equation*}
where $\bu_h,\bv_h$ denote the vectors containing the values of the local DoFs associated to $u_h,v_h\in \VK$and
the stability parameters $\sigma_E$ and $\tau_E$ are two positive constants independent of $h$.

Let us remark that for all numerical tests we will define the eigenvalue $\l_i=\omega_i^2$ where $\omega_i$ is the natural frequency, in a similar way we define by $\l_{h,i}:=\omega_{h,i}^2$ the respective approximations. 
\subsection{Test 1: Rectangular  acoustic cavity}
For this test, the computational domain is a rectangle of the form $\O:=(0,1)\times (0,1.1)$. For this domain, if we consider the physical parameters for the air, the density and the sound speed are $\rho=1 \; \textrm{kg}/\textrm{m}^3$ and $c=340 \; \textrm{m}/\textrm{s}$, the exact eigenvalues and eigenfunctions are know and are of the form 
\begin{align*}
\lambda_{nm}&:=c^2\displaystyle\pi^2\left(n ^2+\left(\frac{m}{1.1} \right)^2 \right),\quad n,m=0,1,2,\ldots, n+m\neq 0,\\
\\
\bu_{nm}(x,y)&:=\begin{pmatrix}\displaystyle n\sin\left(n\pi x\right)\cos\left(\frac{m\pi y}{1.1}\right)\\
\\
\displaystyle \frac{m}{1.1}\cos\left(n\pi x\right)\sin\left(\frac{m\pi y}{1.1}\right)\end{pmatrix}.
\end{align*}

For the discretization of this geometry, we present in Figure \ref{fig:meshesx}  some polygonal meshes considered for the numerical experiments.
\begin{figure}[h]
	\begin{center}
			\centering\includegraphics[height=4.4cm, width=4.2cm]{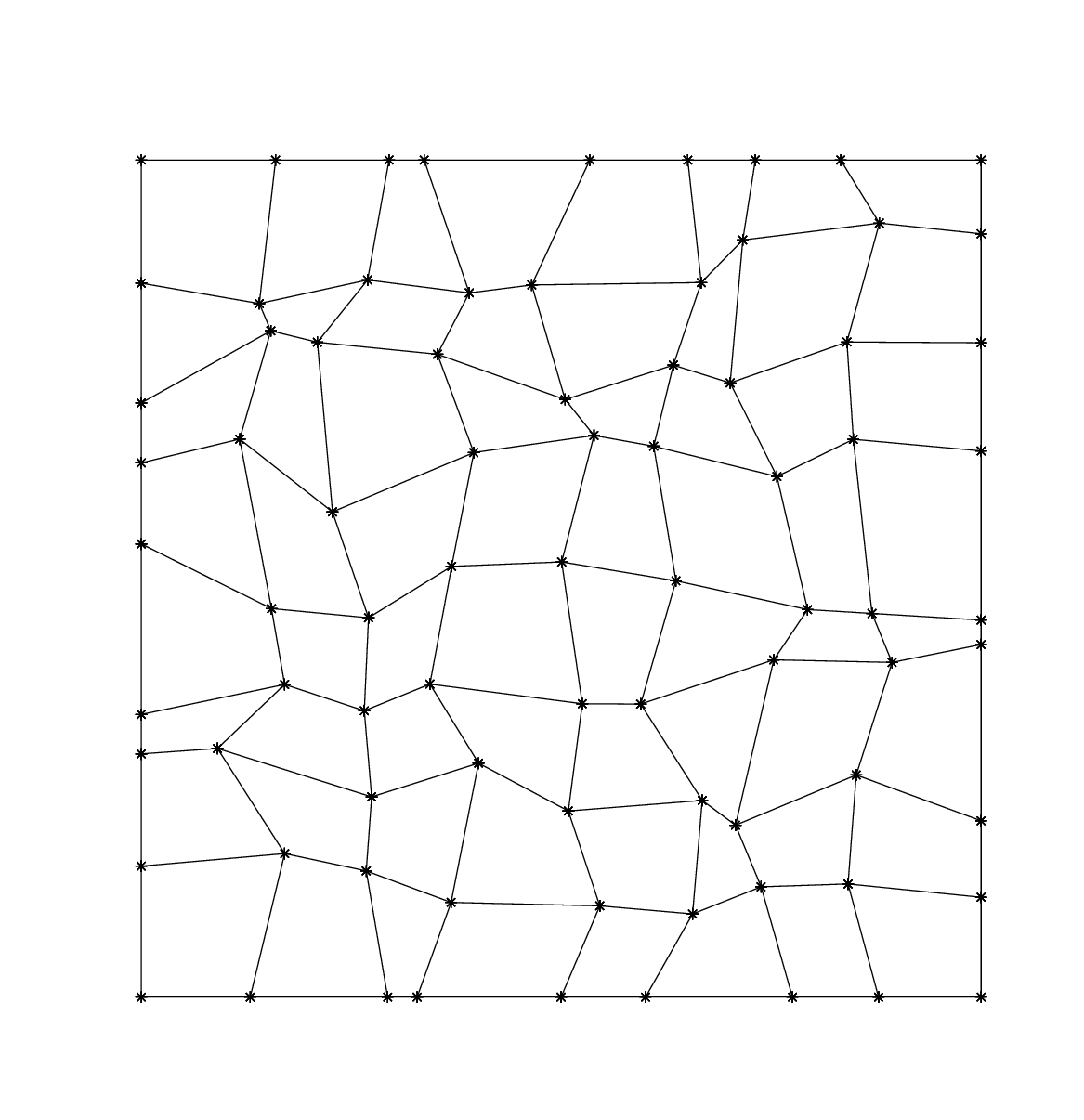}
			\centering\includegraphics[height=4.4cm, width=4.2cm]{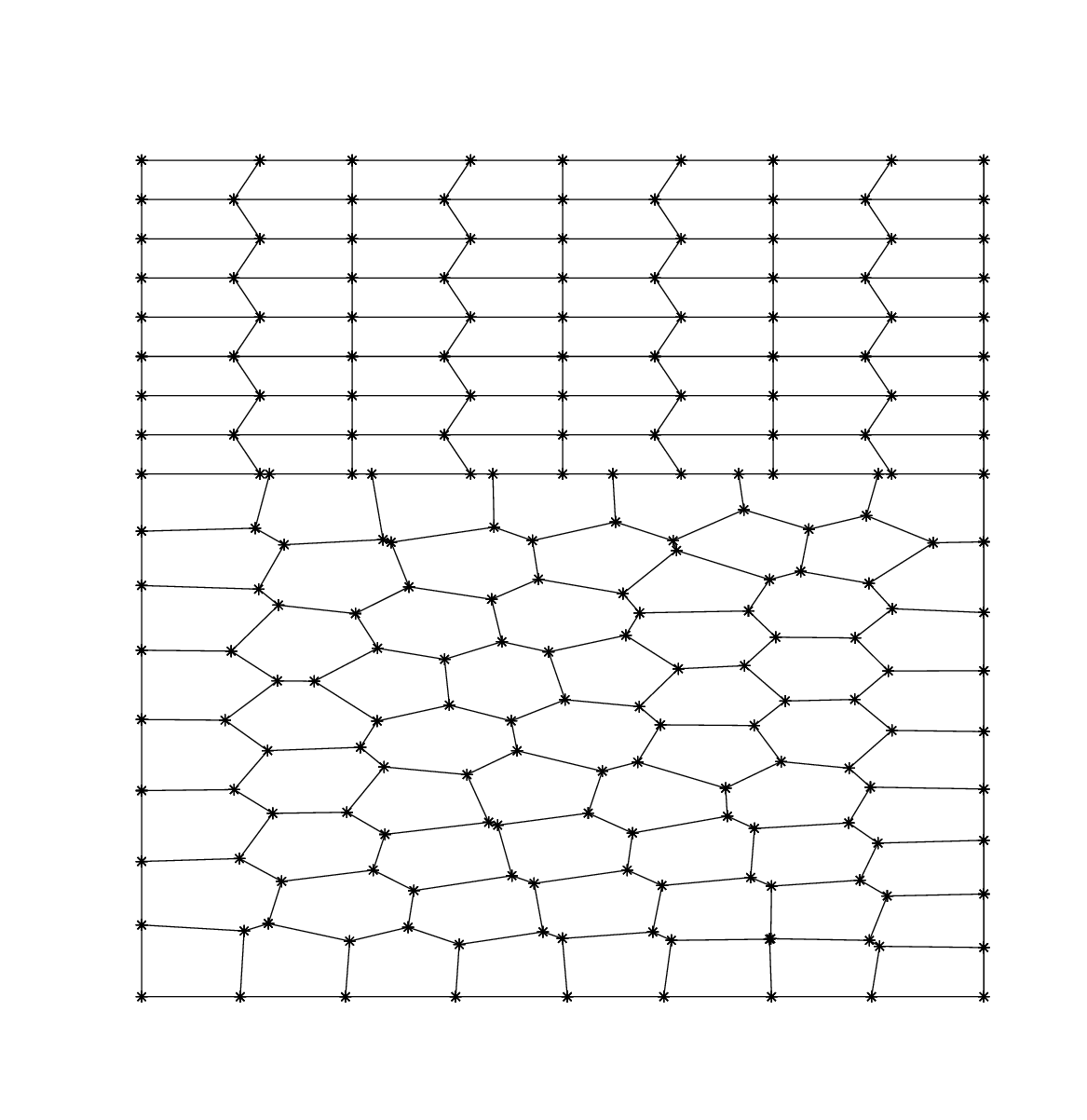}\\
			\centering\includegraphics[height=4.4cm, width=4.2cm]{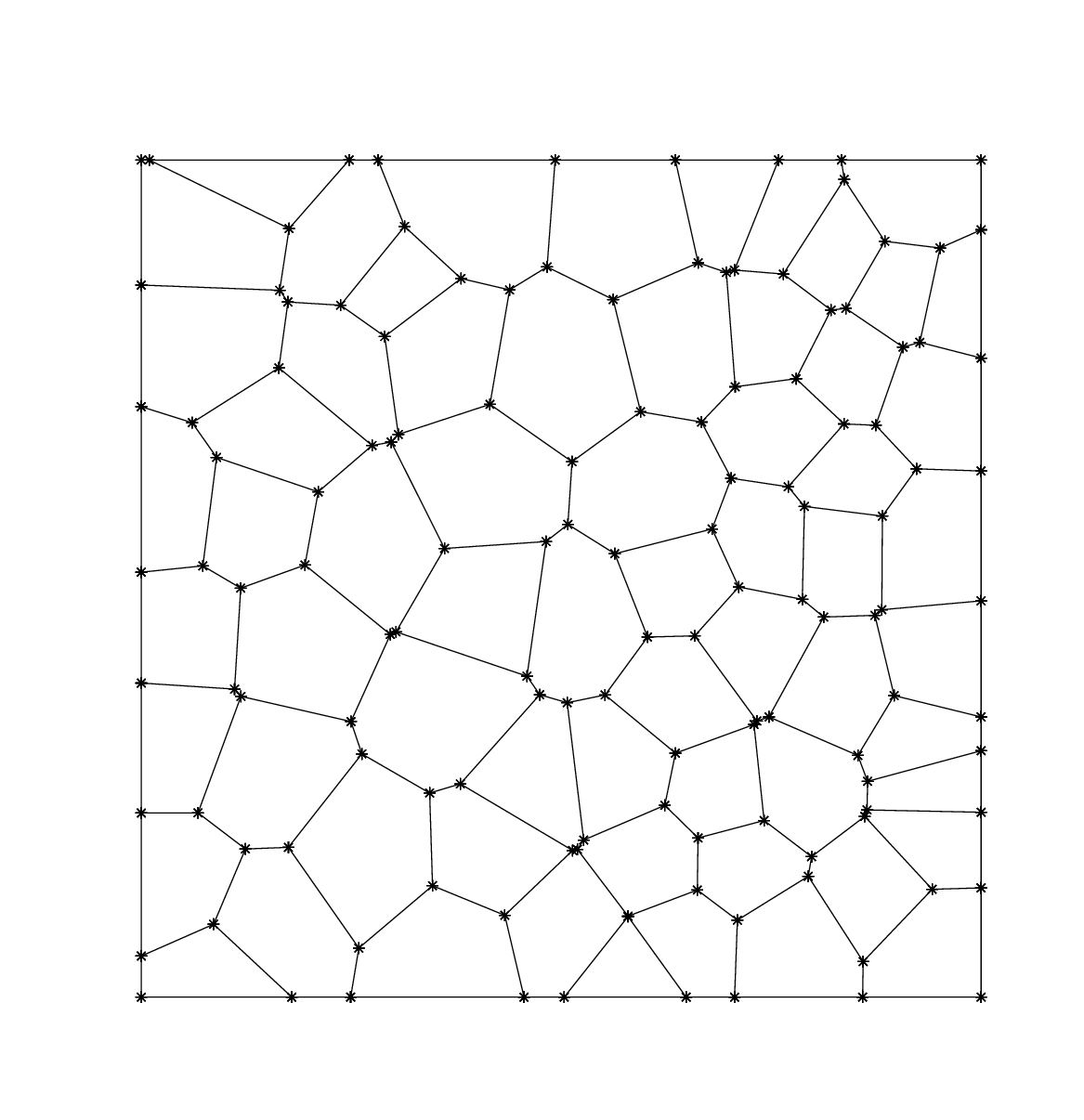}
			\centering\includegraphics[height=4.1cm, width=4.1cm]{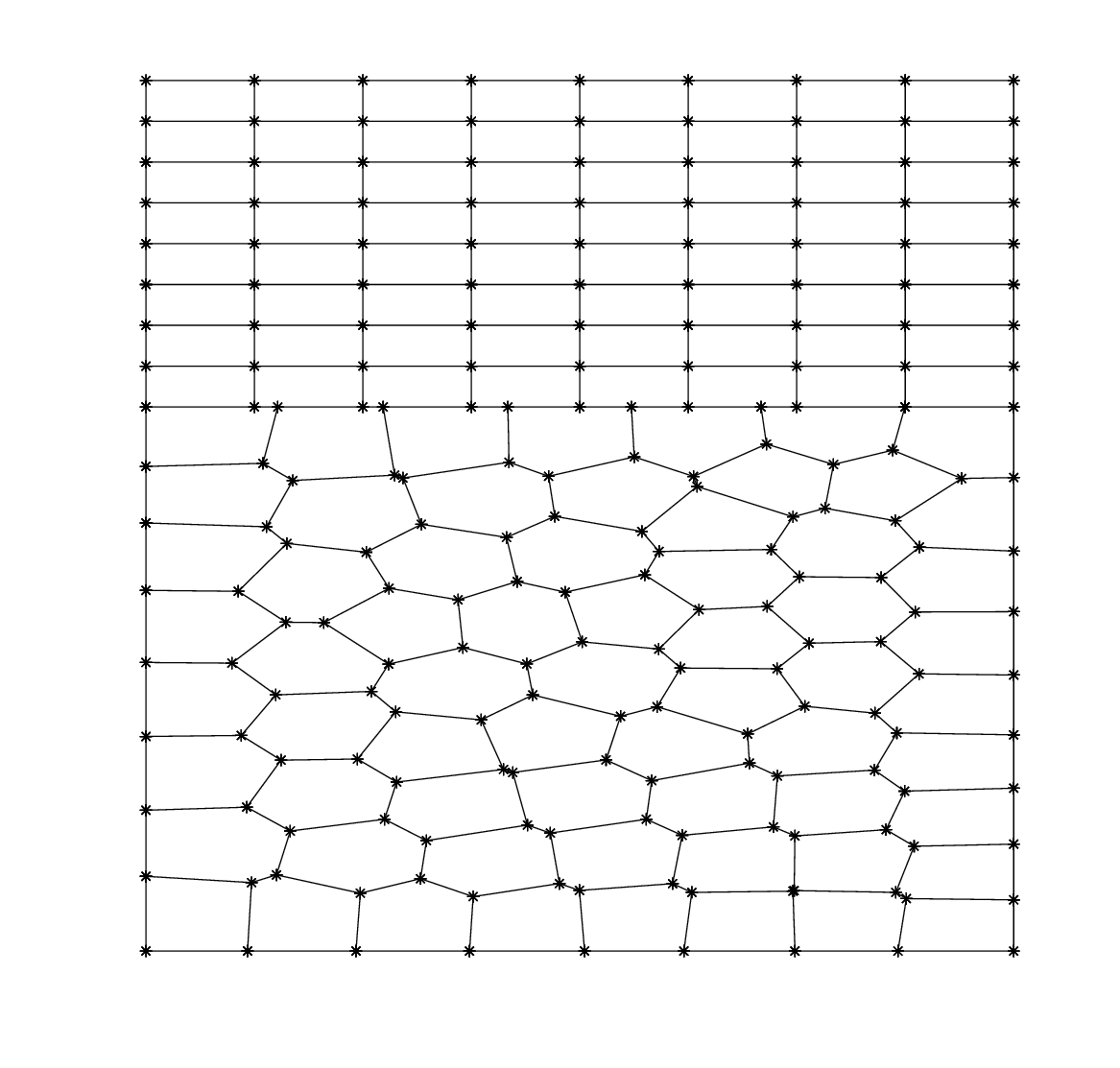}
		\caption{Sample of meshes. Left: $\mathcal{T}_{h}^{1}$ (top left),  $\mathcal{T}_{h}^{2}$ (top right),  $\mathcal{T}_{h}^{3}$ (bottom left) and  $\mathcal{T}_{h}^{4}$ (bottom right)  with $N = 8$.}
		\label{fig:meshesx}
	\end{center}
\end{figure}
Following this, in Figures \ref{fig:error1} and \ref{fig:error2}, error curves are depicted for the initial 4 lowest computational eigenvalues for each family of meshes and different refinement levels. We observe from this error curves a clear quadratic order of convergence.
\begin{figure}[h]
	\begin{center}
			\centering\includegraphics[height=6.4cm, width=6.4cm]{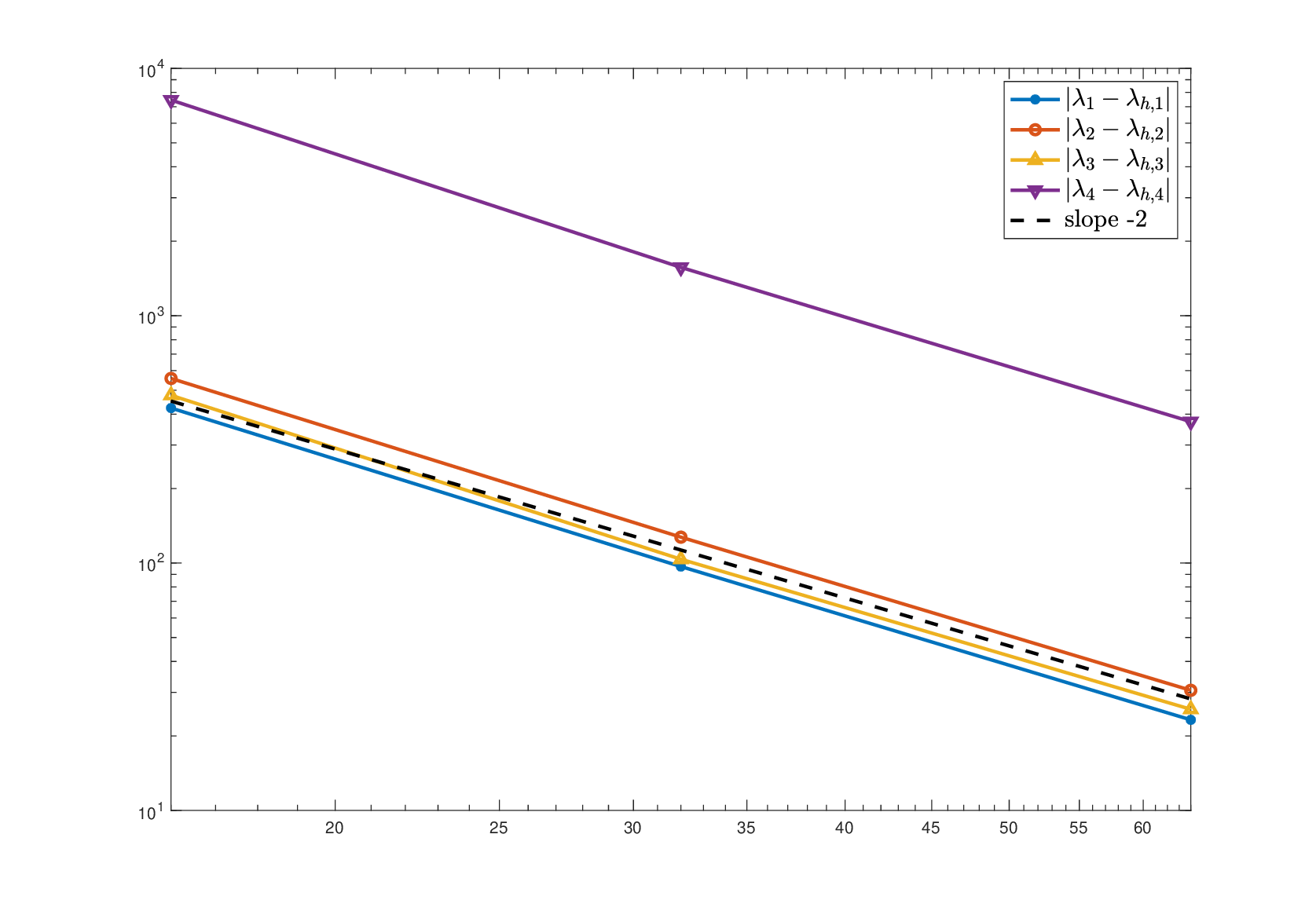}
			\centering\includegraphics[height=6.4cm, width=6.4cm]{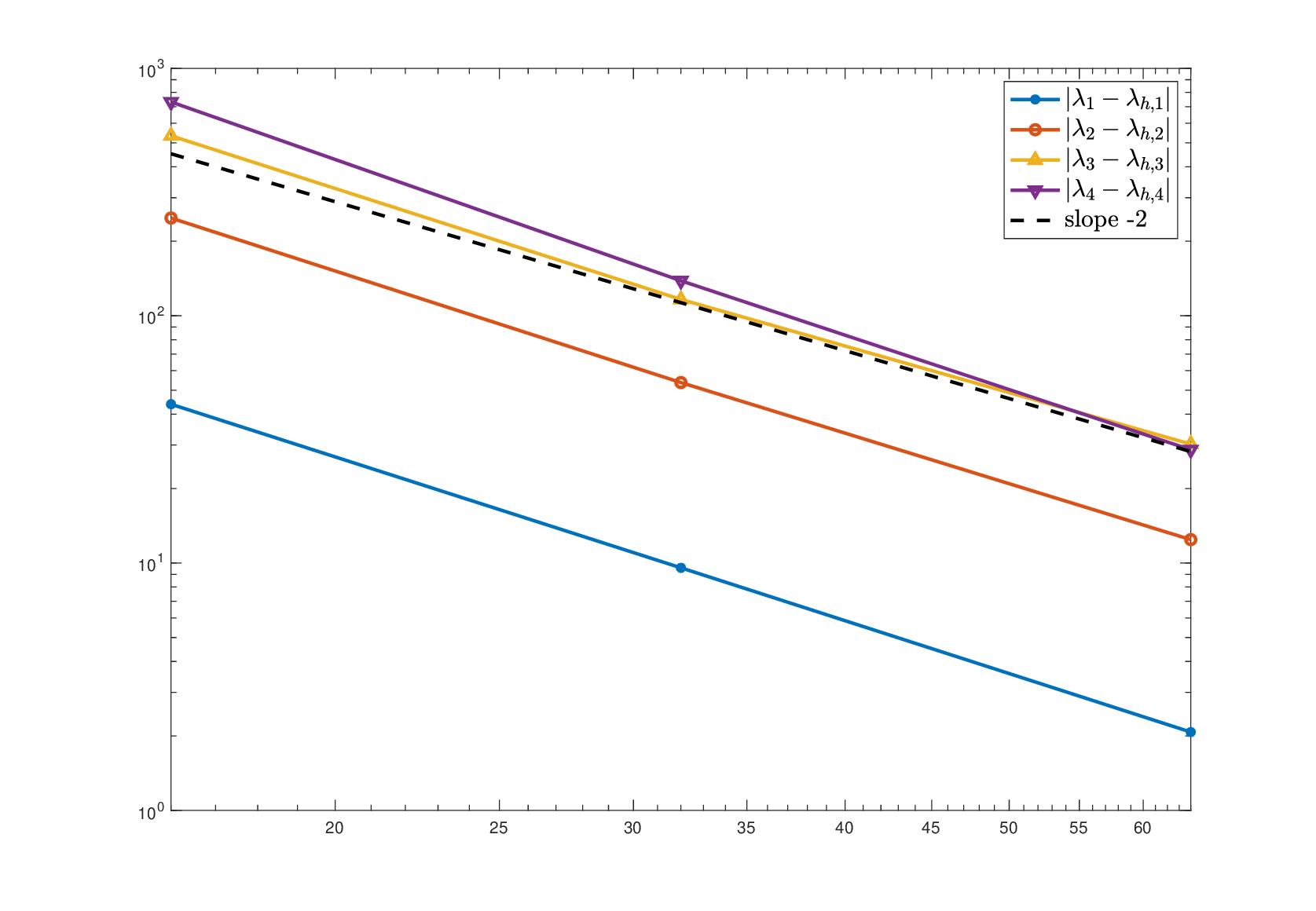}\\
		\caption{Error for  $\mathcal{T}_{h}^{1}$ (left) and  $\mathcal{T}_{h}^{2}$ (right).}
		\label{fig:error1}
	\end{center}
\end{figure}
\begin{figure}[h]
	\begin{center}
			\centering\includegraphics[height=6.4cm, width=6.4cm]{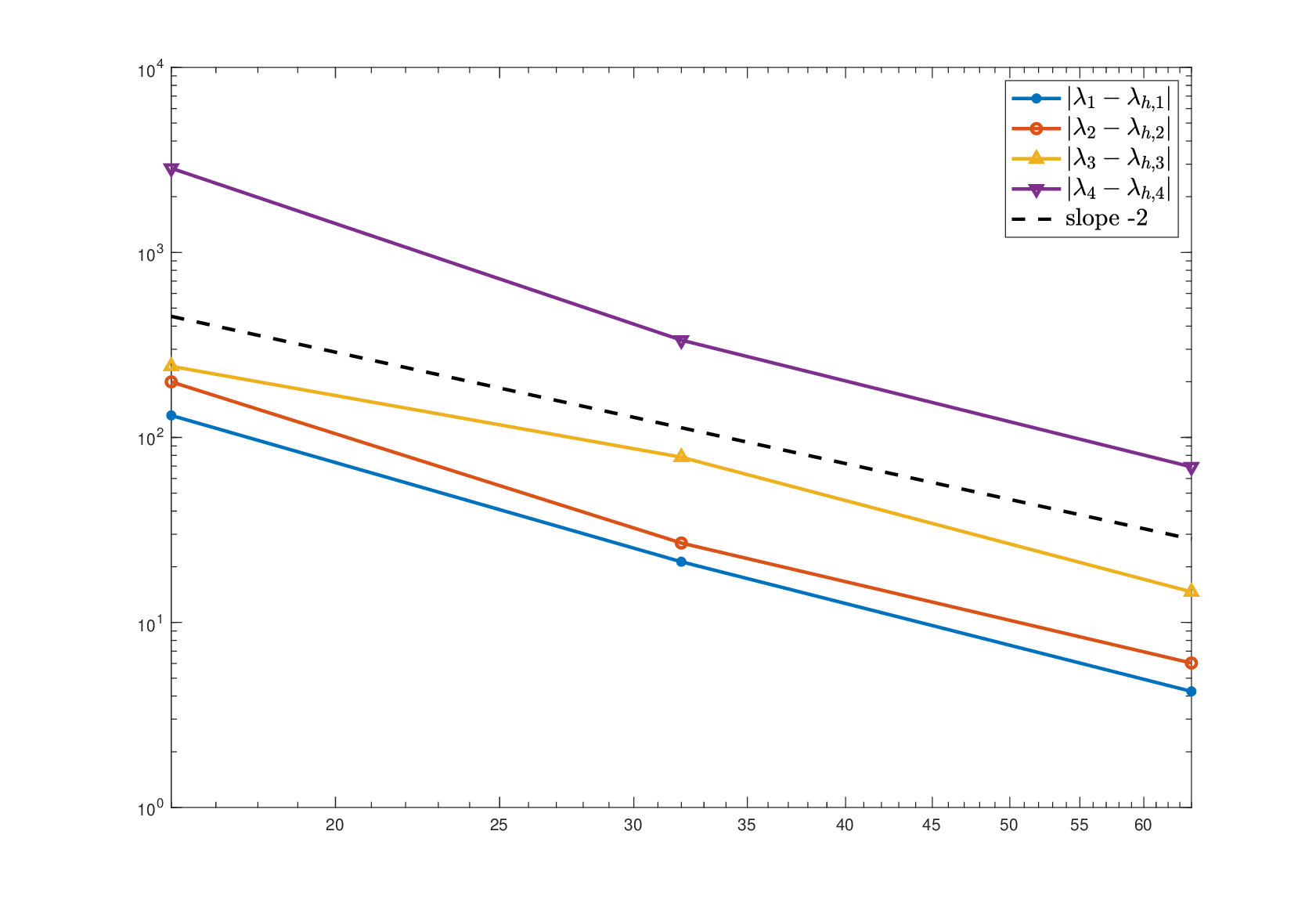}
			\centering\includegraphics[height=6.4cm, width=6.4cm]{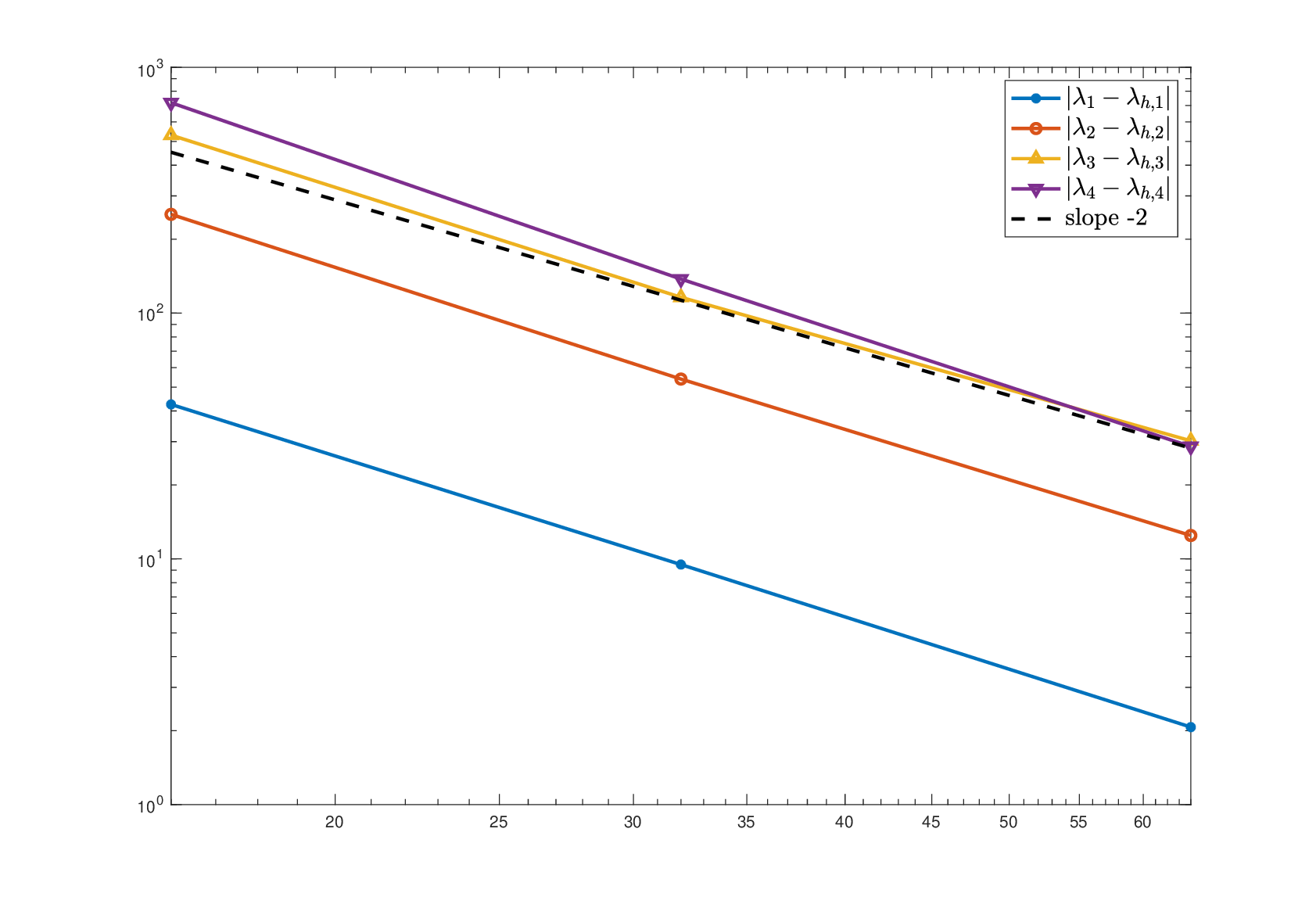}
		\caption{Error for   $\mathcal{T}_{h}^{3}$ (left) and  $\mathcal{T}_{h}^{4}$ (right).}
		\label{fig:error2}
	\end{center}
\end{figure}

Finally, in Figure \ref{fig:p&u} we present the plots of the first, third and fifth eigenfunctions associated to the pressure, and their corresponding displacement fields. Since for the lowest degree ($k=1$) only the DOF information on the sides is available, it is necessary to reconstruct the polynomial projector on each element of the mesh, in order to plot the eigenfunctions and displacements. 
\begin{figure}[h]
	\begin{center}
			\centering\includegraphics[height=4.0cm, width=4.0cm]{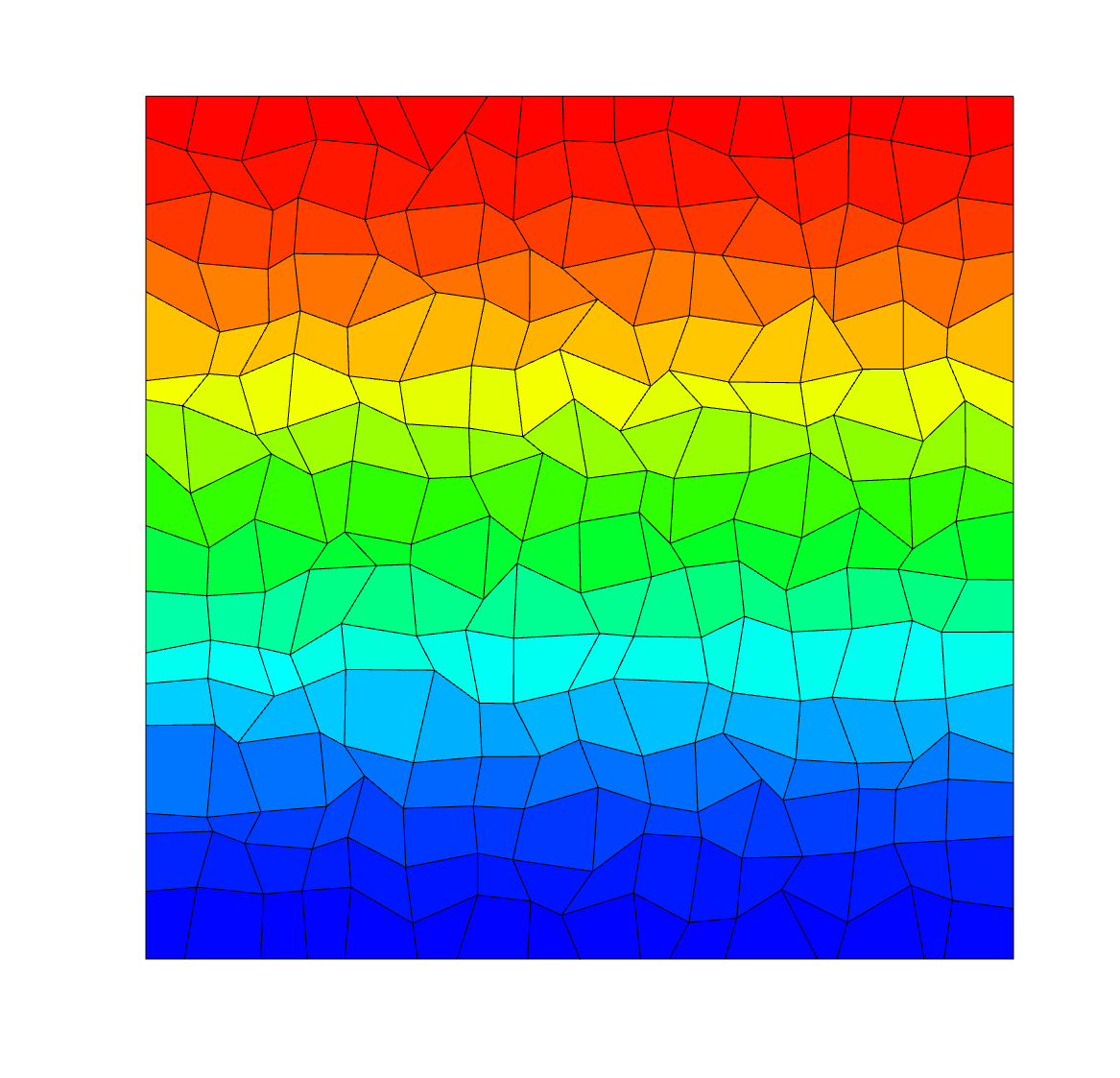}
			\centering\includegraphics[height=4.0cm, width=4.0cm]{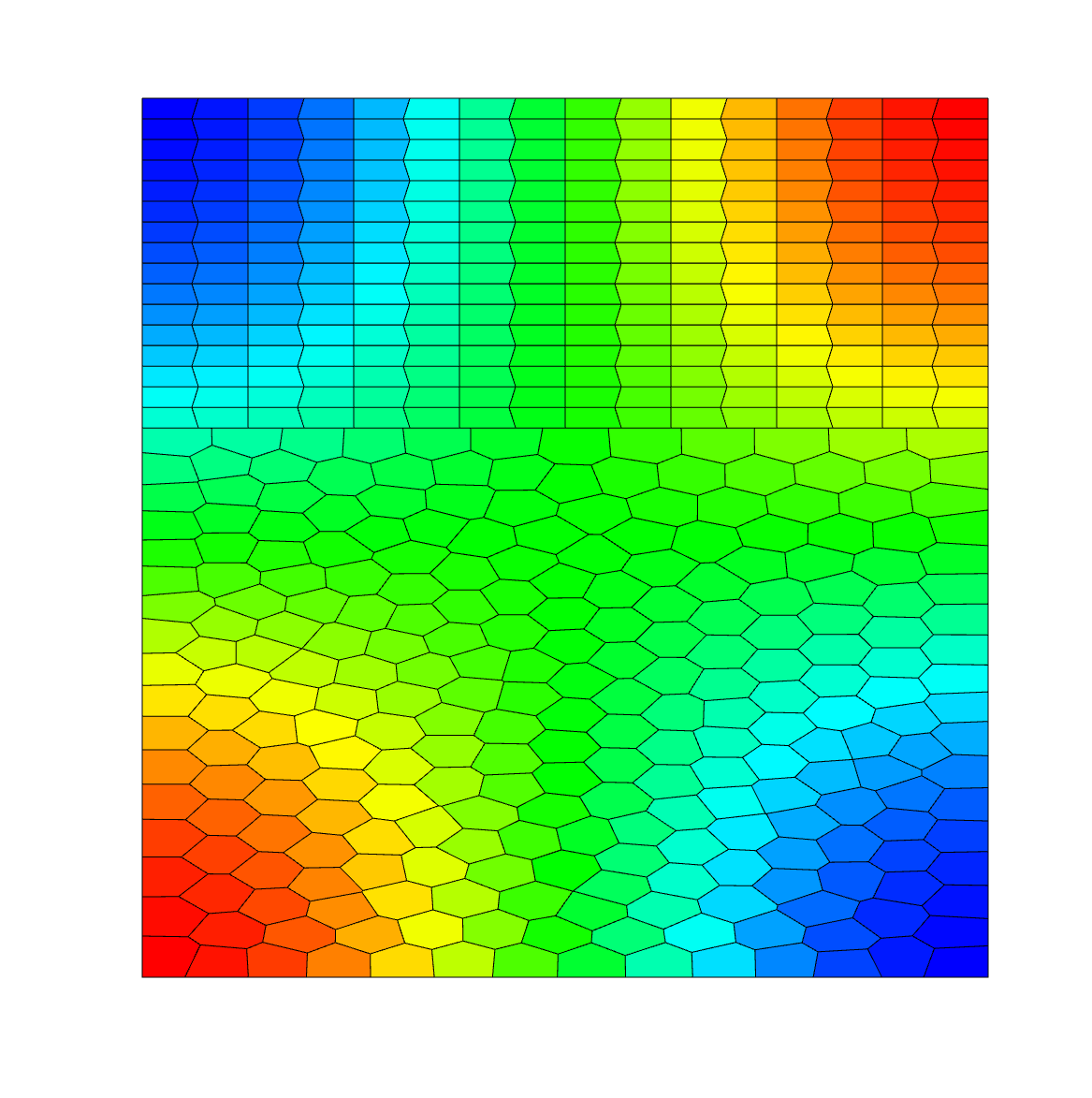}
			\centering\includegraphics[height=4.0cm, width=4.0cm]{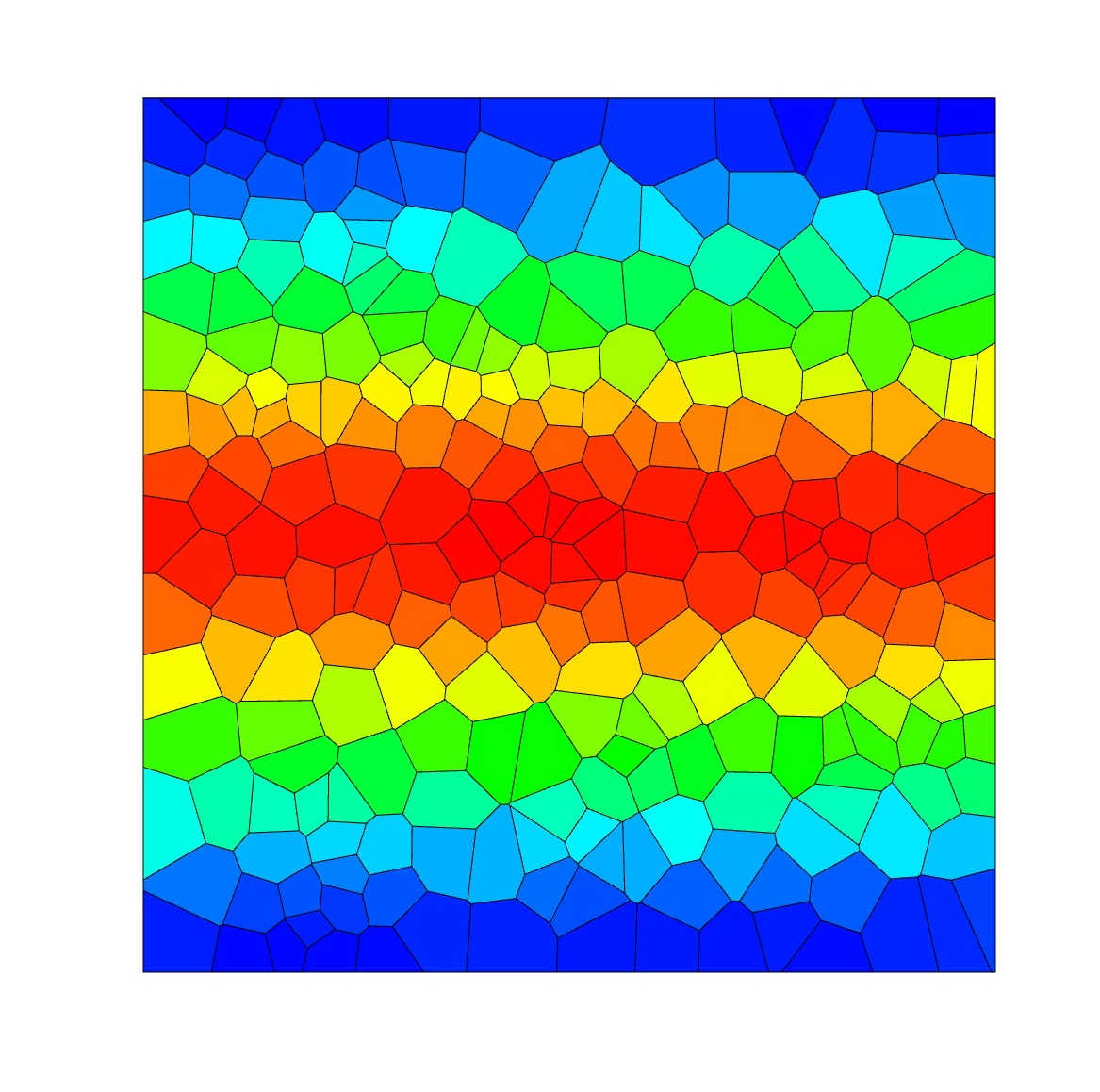}\\
			\centering\includegraphics[height=4.0cm, width=4.0cm]{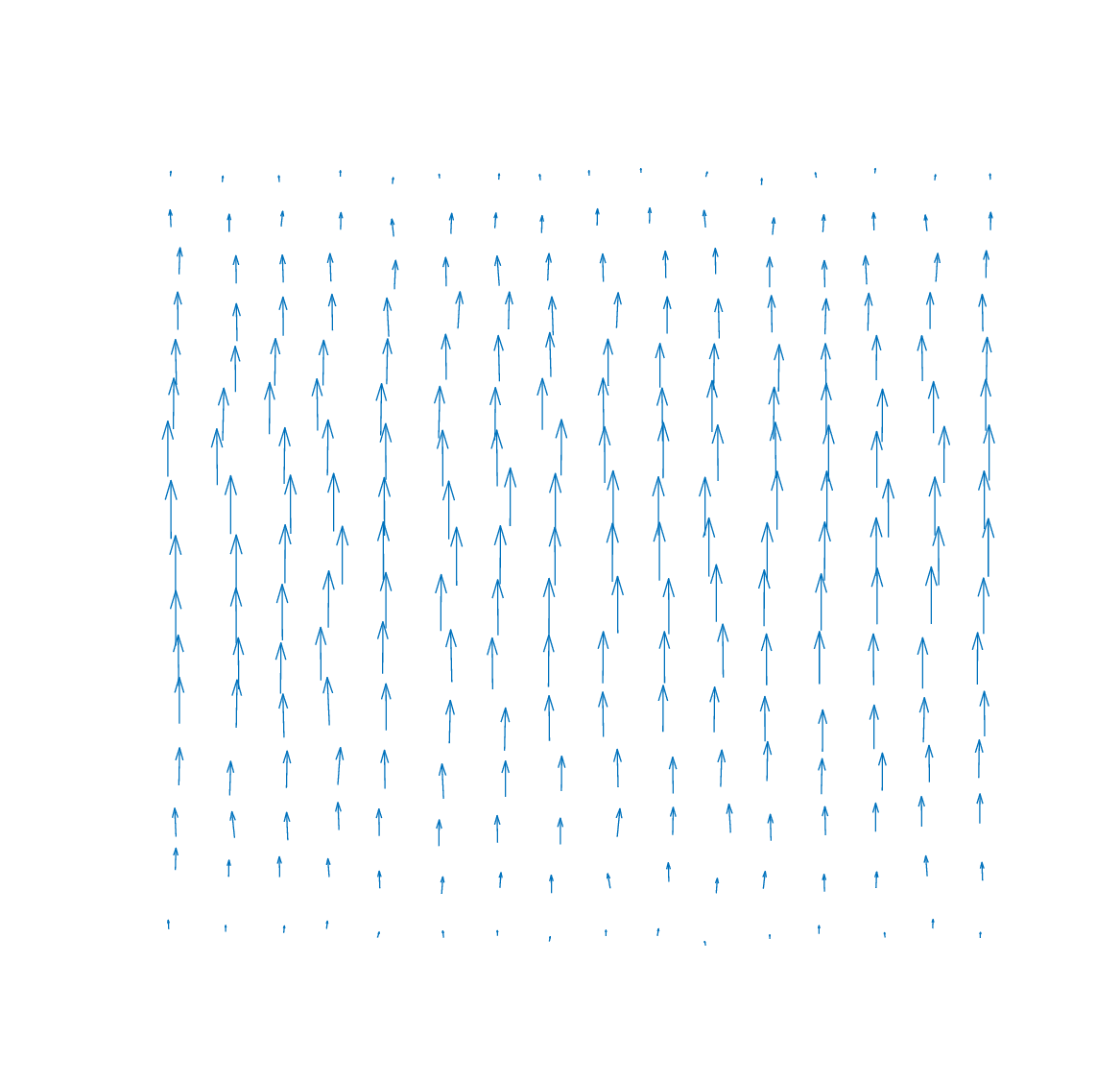}
			\centering\includegraphics[height=4.0cm, width=4.0cm]{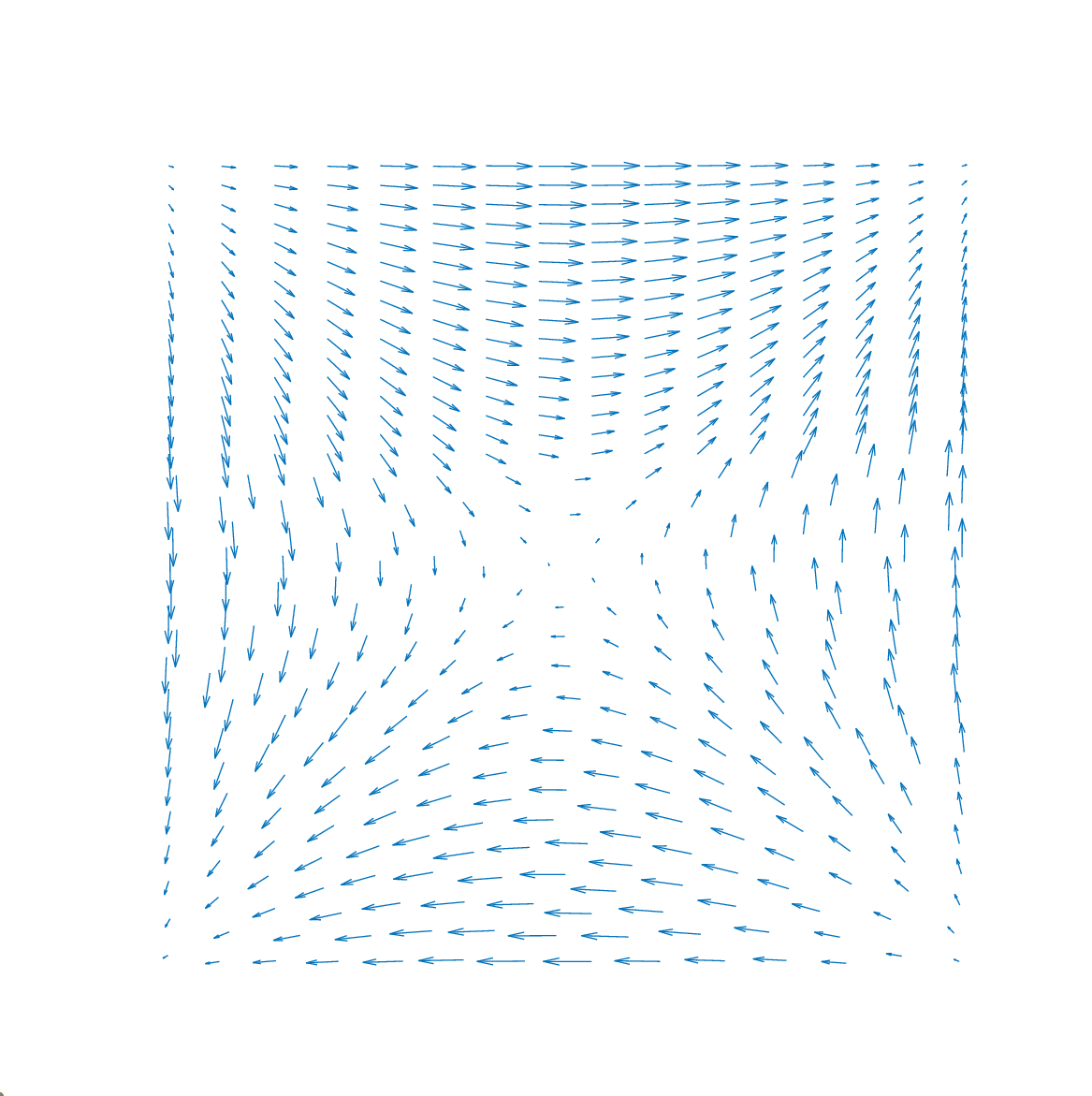}
			\centering\includegraphics[height=4.0cm, width=4.0cm]{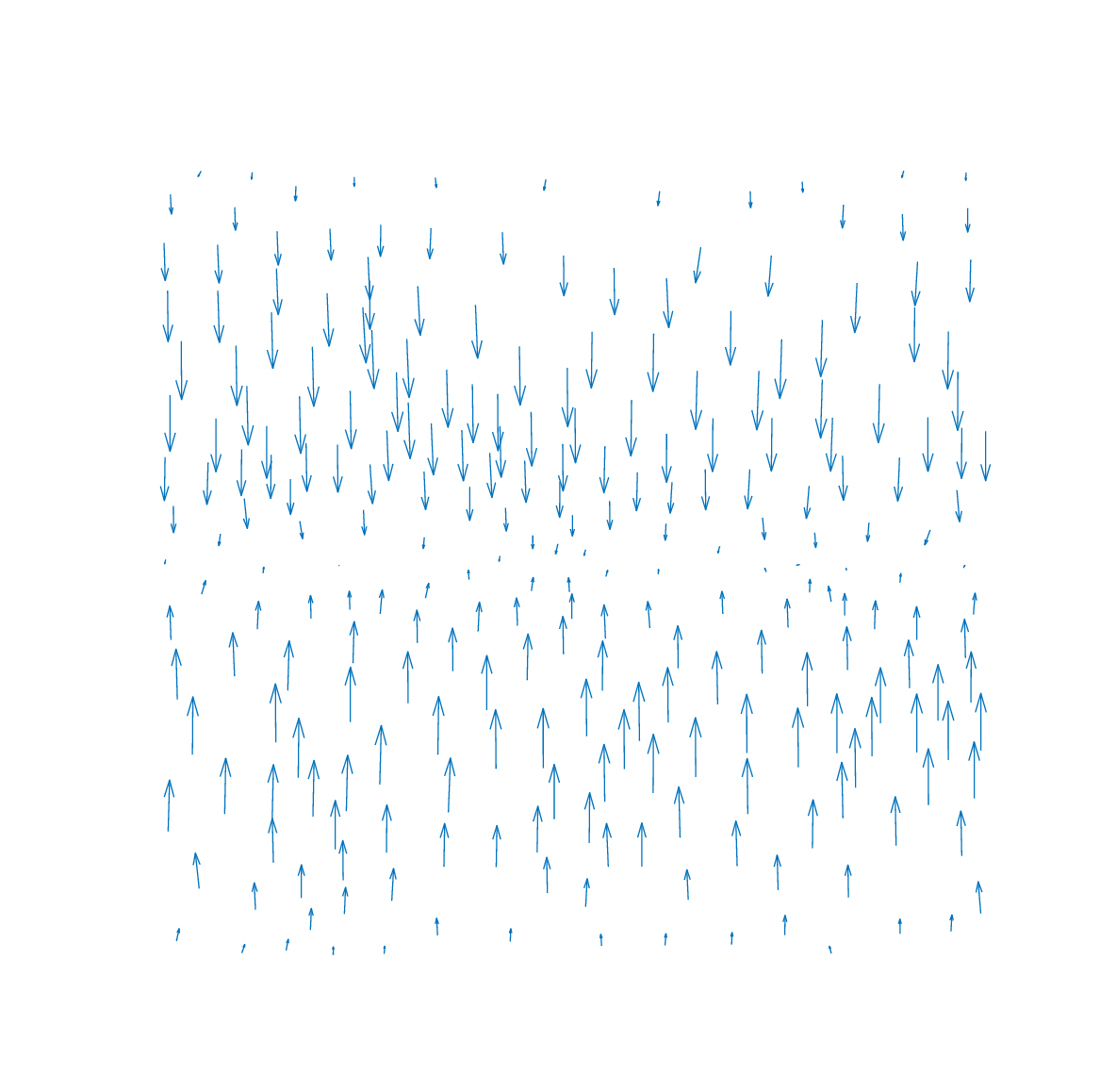}
		\caption{First, third and fifth eigenfunctions on the rectangular acoustic cavity and the corresponding displacements obtained for this test. Top row: $p_{1,h}$, $p_{3,h}$ and $p_{5,h}$; bottom row: corresponding displacement fields $u_{1,h}$, $u_{3,h}$ and $u_{5,h}$.}
		\label{fig:p&u}
	\end{center}
\end{figure}

\subsection{ Test 2: Effects on the stability constants}
The aim of this test is to analyze the influence of the stability constant $\sigma_E$ on
the computed spectrum. To make matters precise, the stabilization, if is not correctly chosen, causes that the method may introduce spurious eigenvalues, as has been observed in, for instance, \cite{MR4658607,MR3340705}, where the conforming VEM has been applied. We expect that a similar behavior is possible to be attained with the NCVEM. For this test, we have considered the rectangle $\O:=(0,1)\times (0,1.1)$ with the boundary condition $\nabla p\cdot\boldsymbol{n}=0$. We report in Tables \ref{tabla2} and \ref{tabla3} the computed eigenvalues, where the numbers inside boxes represent spurious eigenvalues, whereas the rest correspond to physical ones. Let us remark that the results of the aforementioned tables have been obtained with $\CT_h^4$. For the other meshes, the results also hold, in the sense that  spurious eigenvalues appear on the computed spectrum.
The strategy for this tests is as follows: we fix the refinement level in $N=8$ and then we start to move the parameter $\sigma_E$. This is reported in Table \ref{tabla2}.

\begin{table}[h]\label{tabla2}
\caption{ Computed lowest eigenvalues for $\mathcal{T}_{h}^{2}$, $\rho = 1 = c$  and  $4^{-2} \leq \sigma_E \leq 4^{2}$.}
\begin{center}
\begin{tabular}{|c|c|c|c|c|c|c|c|} \hline
$\sigma_E=\dfrac{1}{16}$ & $\sigma_E=\dfrac{1}{4} $ & $\sigma_E=1$ & $\sigma_E=4$ & $\sigma_E=16$ & $\l_{i}$  \\ \hline 
    \fbox{0.4284} &   0.8043   &  0.8240&    0.8276  &  0.8285 & 0.8264\\
    \fbox{0.4642} &0.8875&    0.9873 &   1.0073 &  1.0123& 1.0000\\
    \fbox{0.4774}  &  1.5343&   1.8047  & 1.8364   &1.8441& 1.8264\\
    \fbox{0.4812}   & \fbox{1.7241}   &  3.2412  &   3.3263   &  3.3420& 3.3058\\
    \fbox{0.4894}  &  \fbox{1.8606}   &  3.6909   & 4.1129   & 4.2108& 4.0000\\
    \fbox{0.4904}  &  \fbox{1.9033}   &  4.2085   &  4.3328   &  4.3557& 4.3058\\
    \fbox{0.4937}  &  \fbox{1.9270}   &  4.3992   & 4.9354   &  5.0439&  4.8264\\\hline
\end{tabular}
\end{center}
\end{table}
Table \ref{tabla2} shows the presence of spurious eigenvalues for $\sigma_E =1/16$ and $\sigma_E =1/4$ for the mesh considered with the refinement level $N=8$.  Moreover, when the values of $\sigma_E =1$, the pollution of the spectrum starts to vanish. This fact gives us the clue that for $\sigma_E \geq 1$ , the proposed method provides the physical modes for the acoustics eigenvalue problem. 
Now, we are interested to know how the spurious eigenvalues behave when the mesh is refined. To do this task, let us focus when  $ \sigma_E=1/16$ and then we begin with the refinement process. These results are reported in Table \ref{tabla3}.
    \begin{table}[h]\label{tabla3}
\caption{ Computed lowest eigenvalues for  $\mathcal{T}_{h}^{2}$, $\rho = 1 = c$ and   $\sigma_E =1/16$.}
\begin{center}
\begin{tabular}{|c|c|c|c|c|c|c|c|} \hline
      $N=8$ &   $N=16$ &  $N=32$ &  $N=64$ & $\l_{i}$\\\hline
    \fbox{0.4284} &   0.8079 &  0.8230 &0.8257& 0.8264\\
    \fbox{0.4642}  &0.8940  &  0.9798& 0.9954& 1.0000\\
    \fbox{0.4774}  & 1.5221  &  1.7880&1.8178&  1.8264 \\
    \fbox{0.4812}& \fbox{1.7760} &   3.2325 & 3.2930& 3.3058\\
    \fbox{0.4894} & \fbox{1.7895} &   3.5937&3.9213&  4.0000\\
    \fbox{0.4904}  & \fbox{1.8017} &   4.1615&4.2818&4.3058 \\
    \fbox{0.4937}   & \fbox{1.8251} &   4.2866&4.7287&4.8264  \\\hline
       \end{tabular}
\end{center}
\end{table}

 It can be seen from  Table \ref{tabla3}. that the spurious values disappear as the meshes are refined. This analysis suggests that the way of minimizing this risk is to take $ \sigma_E \geq 1$ and sufficiently refined meshes.    
    
\subsection{Test 3: Circular ring-shaped domain }
For the following  test, we have considered a curved domain, whose definition is given by  
$$\Omega:=\left\{(x,y)\in\mathbb{R}^2: \dfrac{1}{2}\leq x^2+y^2\leq 2\right\},$$ 
with $\nabla p\cdot \boldsymbol{n}=0$ on the boundary $\partial\O$. In this test the refinement parameter $N$, represents the number of elements intersecting the boundary. Voronoi and triangle meshes with deformed midpoints were used for this experiment (see Figure \ref{fig:meshecirc}). In Table \ref{tabla4} we present the four lowest eigenvalues $\l_{h,i}$. The table also includes the estimated orders of convergence. 

\begin{figure}[h]
	\begin{center}
			\centering\includegraphics[height=4.4cm, width=6.8cm]{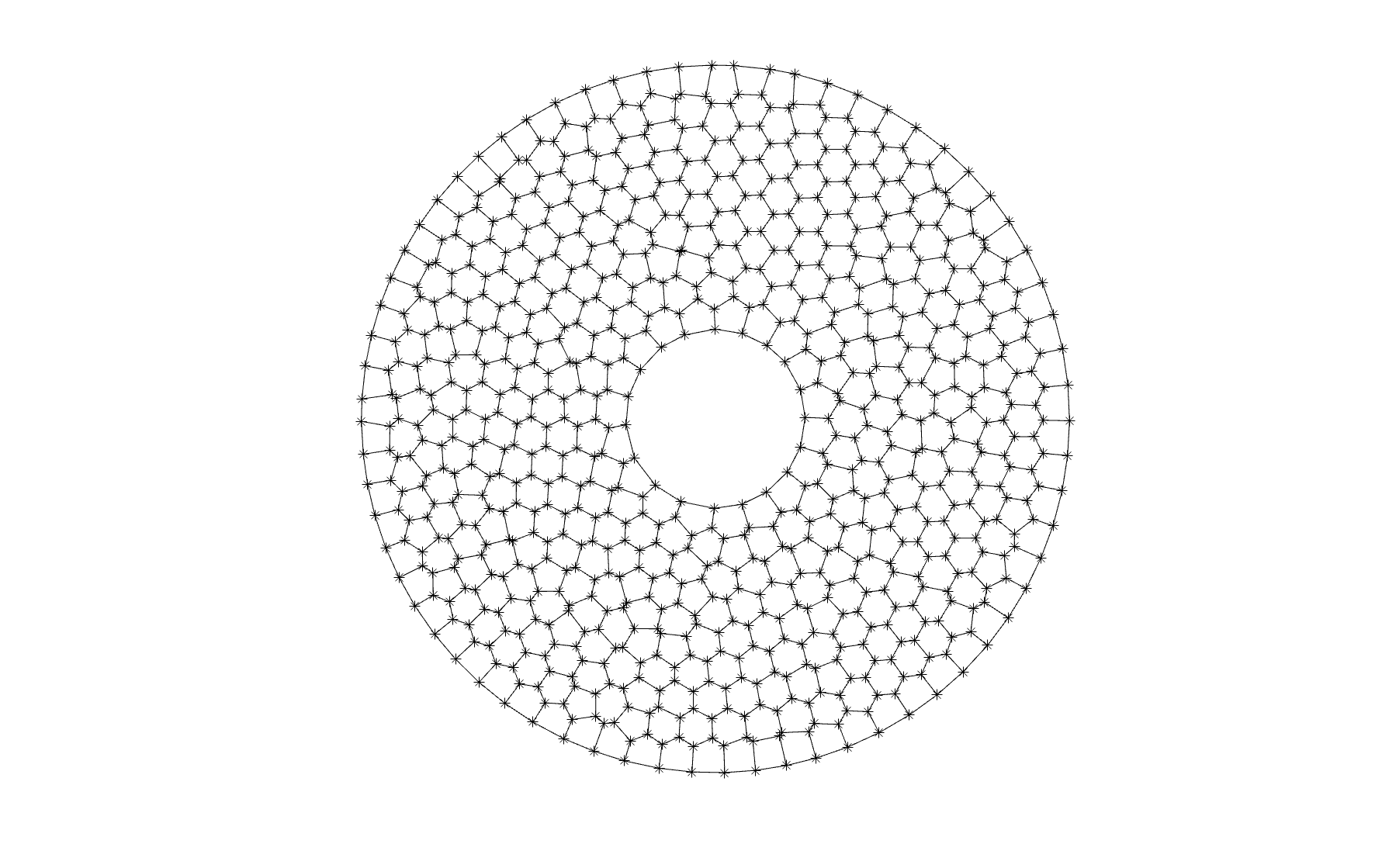}
			\centering\includegraphics[height=4.4cm, width=4.4cm]{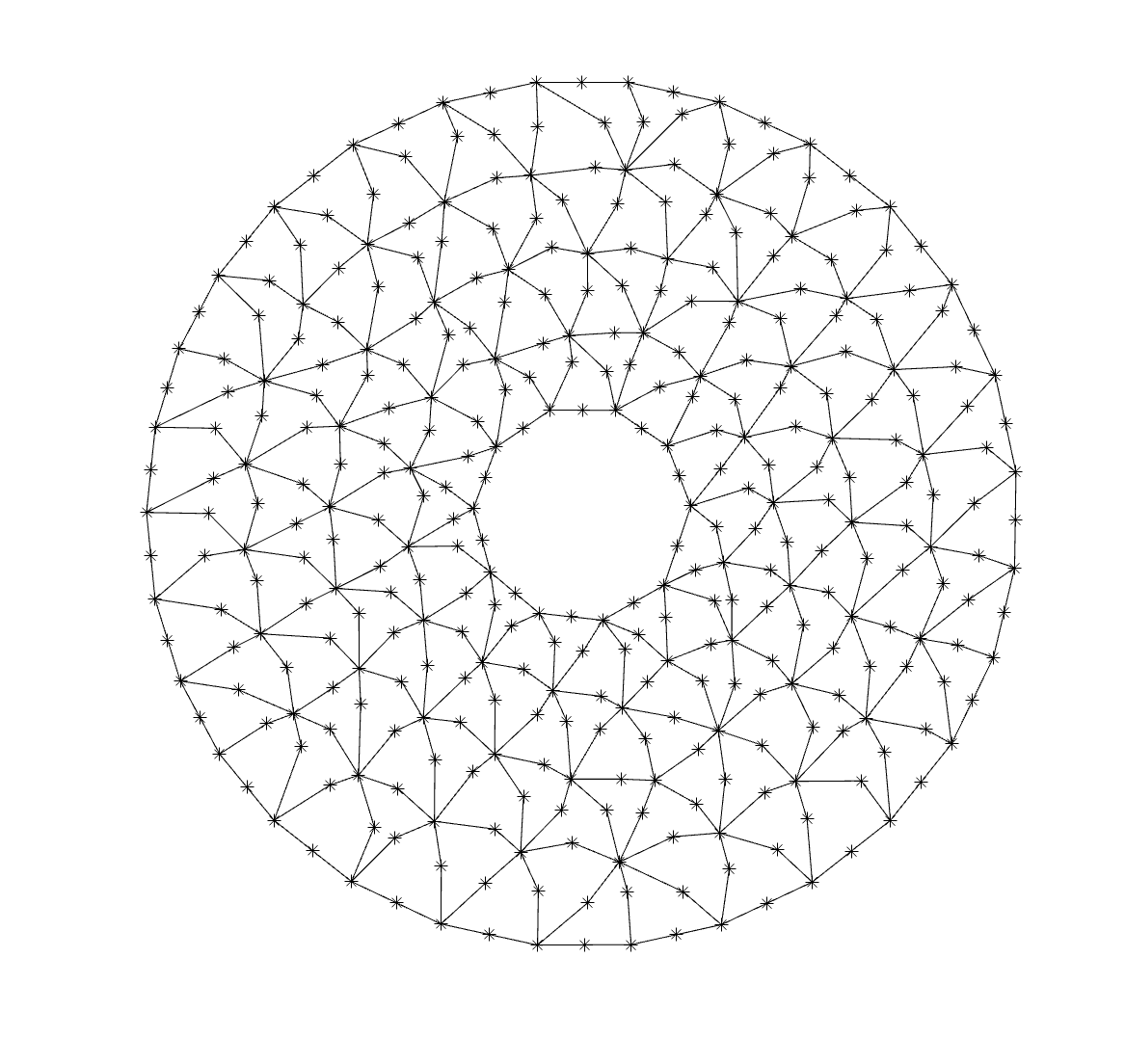}
		\caption{Sample of meshes. Left: $\mathcal{T}_{h}^{C_1}$ (left),  $\mathcal{T}_{h}^{C_2}$,   with $N = 87$ and $N = 78$, respectively.}
		\label{fig:meshecirc}
	\end{center}
\end{figure}

 \begin{table}[h]
\caption{ The lowest computed eigenvalues $\l_{h,i}$, $1\leq i\leq 4$ for  $\mathcal{T}_h^3$.}
\label{tabla4}
\begin{center}

\begin{tabular}{|c|c|c|c|c|c|c|c|} \hline
\multicolumn{8}{ |c| }{$\mathcal{T}_h^{C_1}$} \\ \hline 
 $\l_{h,i}$ & $N = 97$ & $N = 169$& $N = 236$ & $N = 258$ & $N = 295$ &Order & Extr. \\ \hline 
  $\l_{1,h}$ &  0.6714 &   0.6744 &   0.6753  &  0.6754 &   0.6755  &  1.8600   & 0.6762\\
 $\l_{2,h}$ &   0.6716  &  0.6745  &  0.6753  &  0.6754  &  0.6756  &  1.8300   & 0.6762\\
$\l_{3,h}$ &    2.2554  &  2.2614  &  2.2627  &  2.2629  &  2.2632  &  2.1400  &  2.2640\\
$\l_{4,h}$ &    2.2558  &  2.2615  &  2.2628  &  2.2630  &  2.2632  &  2.1200  &  2.2640\\\hline\hline
\multicolumn{8}{ |c| }{$\mathcal{T}_h^{C_2}$} \\ \hline 
 $\l_{h,i}$ & $N = 76$ & $N = 138$& $N = 238$ & $N = 266$ & $N = 294$ &Order & Extr. \\ \hline 
$\l_{h,1}$ & 0.6818  &  0.6778  &  0.6768  &  0.6767  &  0.6766   & 2.1400 &   0.6763\\
$\l_{h,2}$ &    0.6822 &   0.6781 &   0.6769  &  0.6767   & 0.6766  &  1.8400  &  0.6761\\
$\l_{h,3}$ &    2.2885 &   2.2721 &   2.2667 &   2.2662   & 2.2658 &   1.7900&    2.2635\\
$\l_{h,4}$ &    2.2906  &  2.2724 &   2.2669  &  2.2663  &  2.2659 &   1.9100&    2.2639\\ \hline
    \end{tabular}
\end{center}
\end{table}

Once again, a quadratic order of convergence can be clearly appreciated from Table \ref{tabla4}.
In Figures \ref{fig_cir} and \ref{fig_cir2} we present plots for the first and third  eigenfunctions  corresponding to pressures and displacement for the mesh families used. 
\begin{figure}[h]
	\begin{center}
			\centering\includegraphics[height=4.0cm, width=4.0cm]{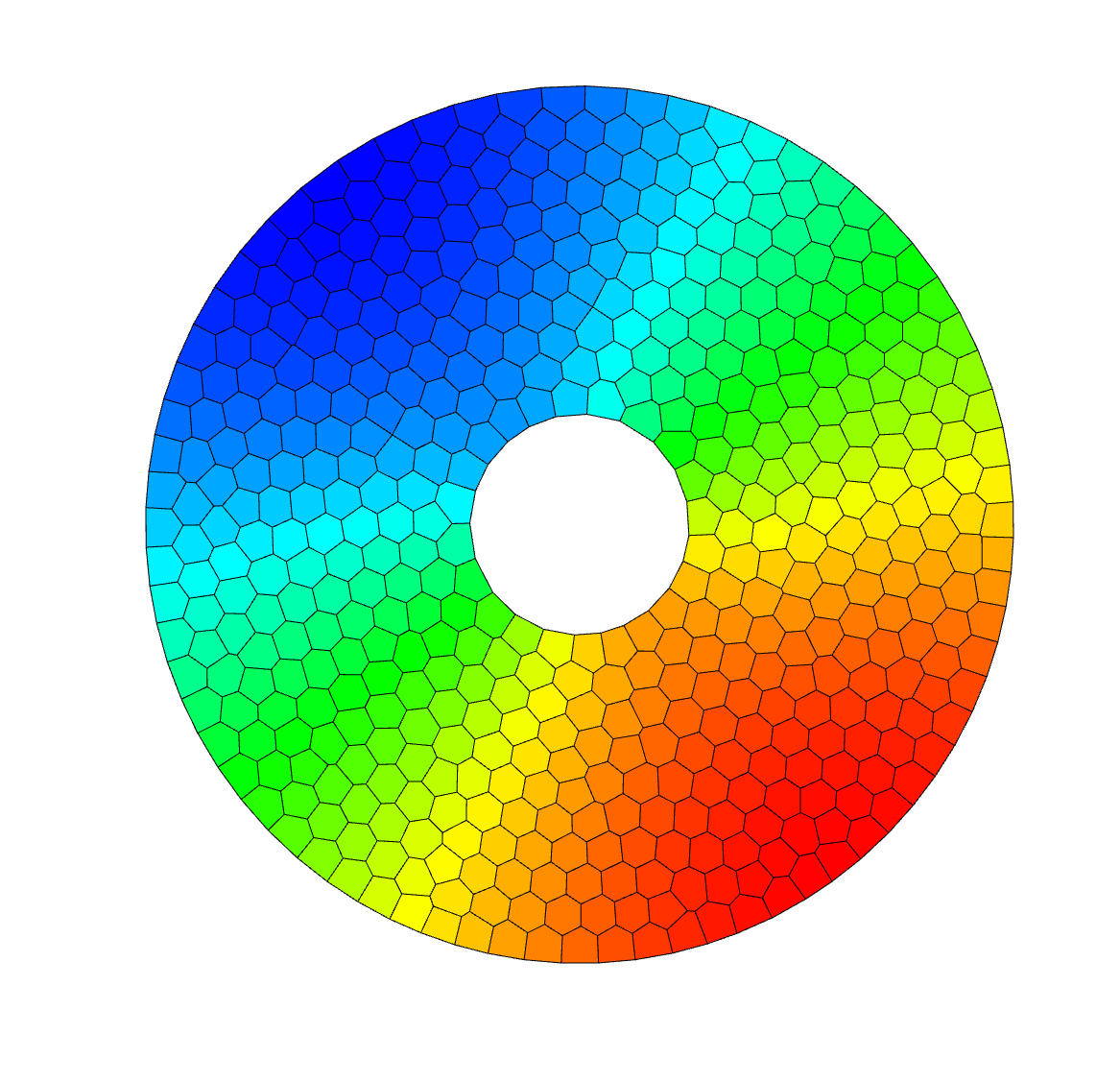}
			\centering\includegraphics[height=4.0cm, width=4.0cm]{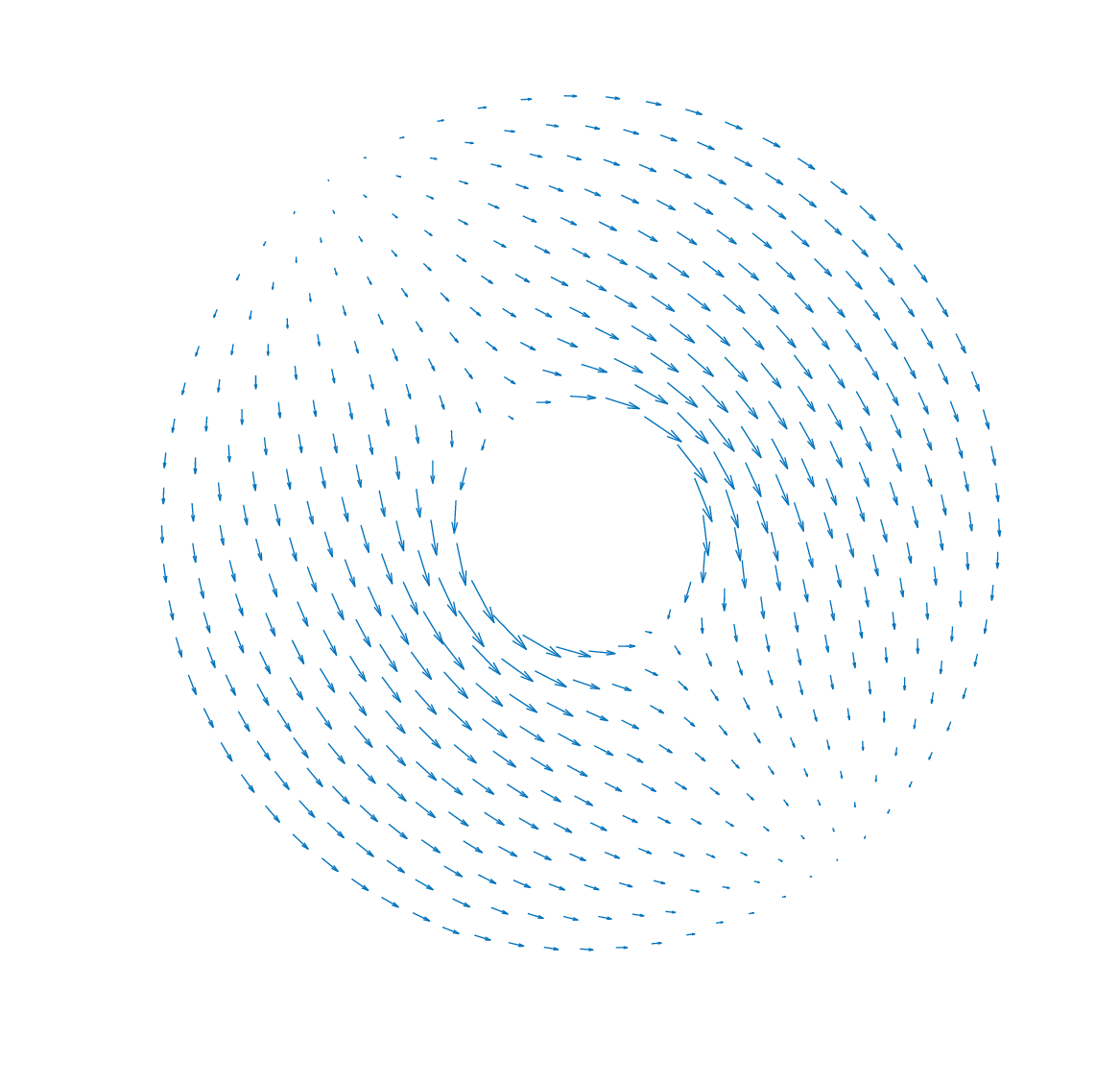}
		\caption{First  eigenfunctions  corresponding to pressures $p_{1,h}$ and displacement  $u_{1,h}$ for $\mathcal{T}_h^{C_1}$.}
		\label{fig_cir}
	\end{center}
\end{figure}
\begin{figure}[h]
	\begin{center}
			\centering\includegraphics[height=4.0cm, width=4.0cm]{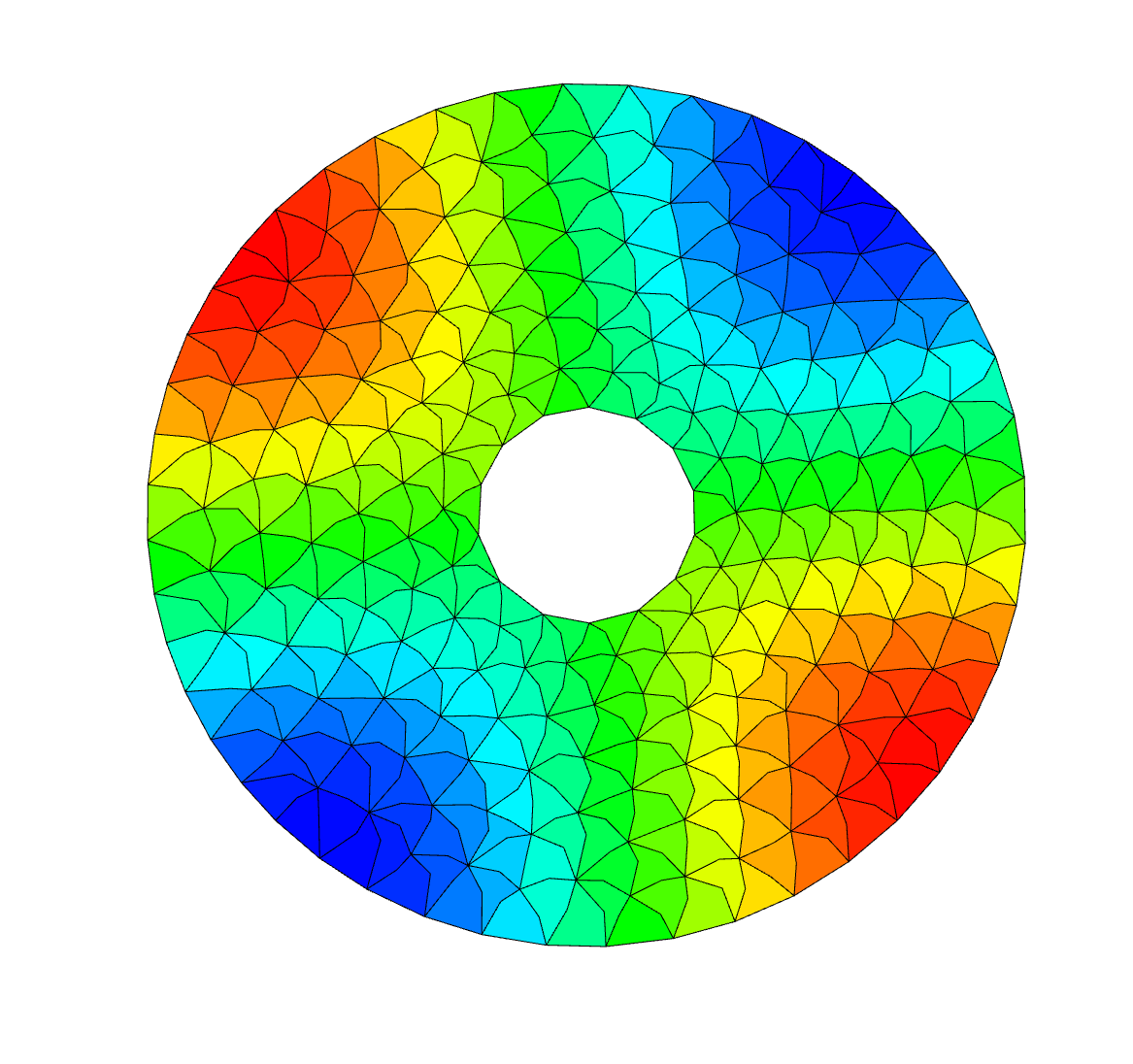}
			\centering\includegraphics[height=4.0cm, width=4.0cm]{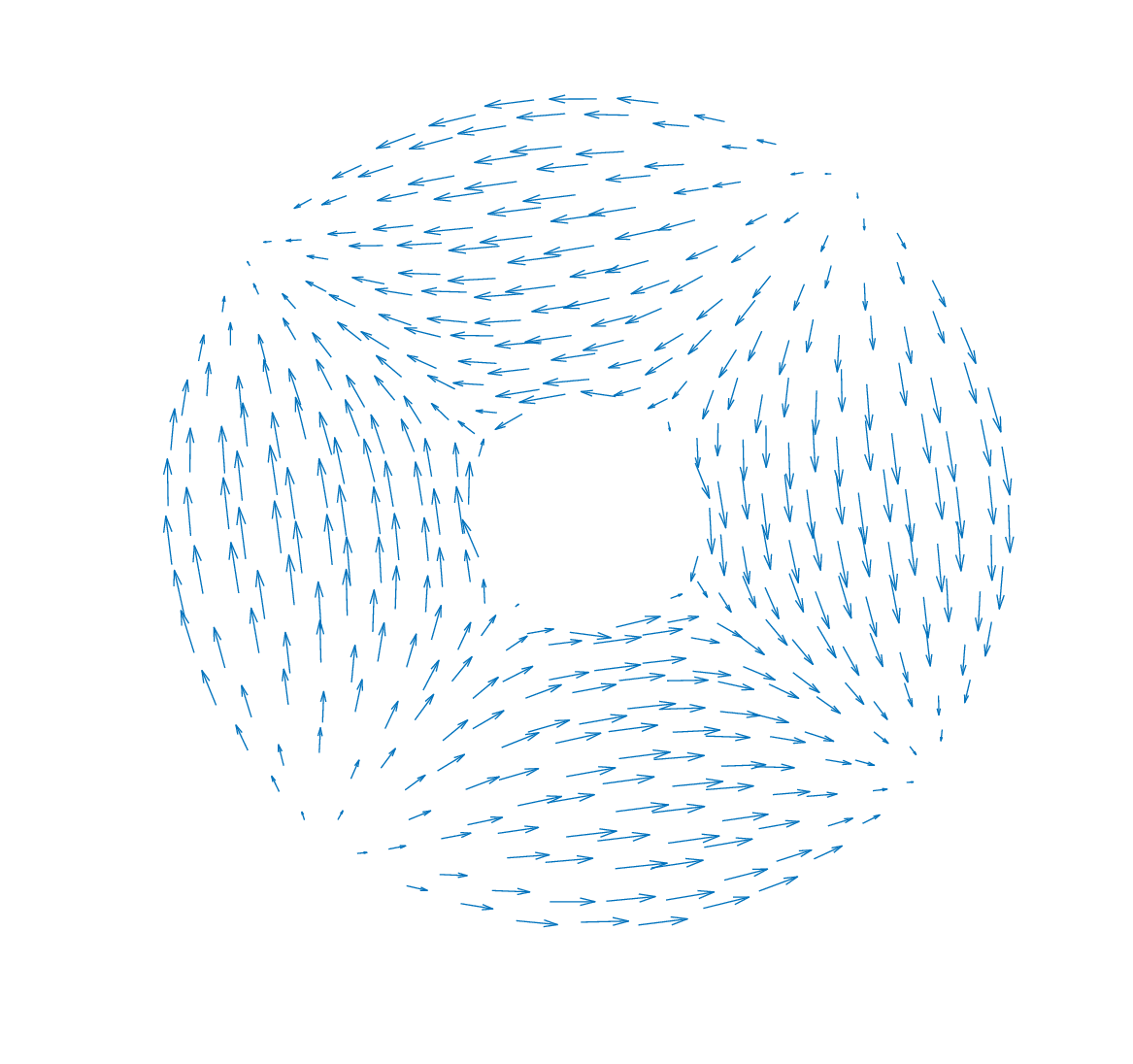}
		\caption{Third  eigenfunctions  corresponding to pressures $p_{h,3}$ and displacement  $u_{h,3}$ for $\mathcal{T}_h^{C_2}$.}
		\label{fig_cir2}
	\end{center}
\end{figure}

\subsection{Test 4: L-shaped domain}

This experiment considers a non-convex domain. We set 
$$\Omega := (-1, 1) \times (-1, 1) \setminus ([-1, 0] \times [-1, 0]),\,\, \text{with}\,\,\nabla p\cdot \boldsymbol{n}=0 \text{ on } \Gamma, $$
with corresponds to a 2D L-shaped domain. In this test, the physical constants used were those of water, i.e: $\rho=1000 kg/m^3$ and $c=1430 m/s$. 
The family of meshes to be used is shown in Figure \ref{fig:meshesL}.

\begin{figure}[h]
	\begin{center}
			\centering\includegraphics[height=4.4cm, width=4.2cm]{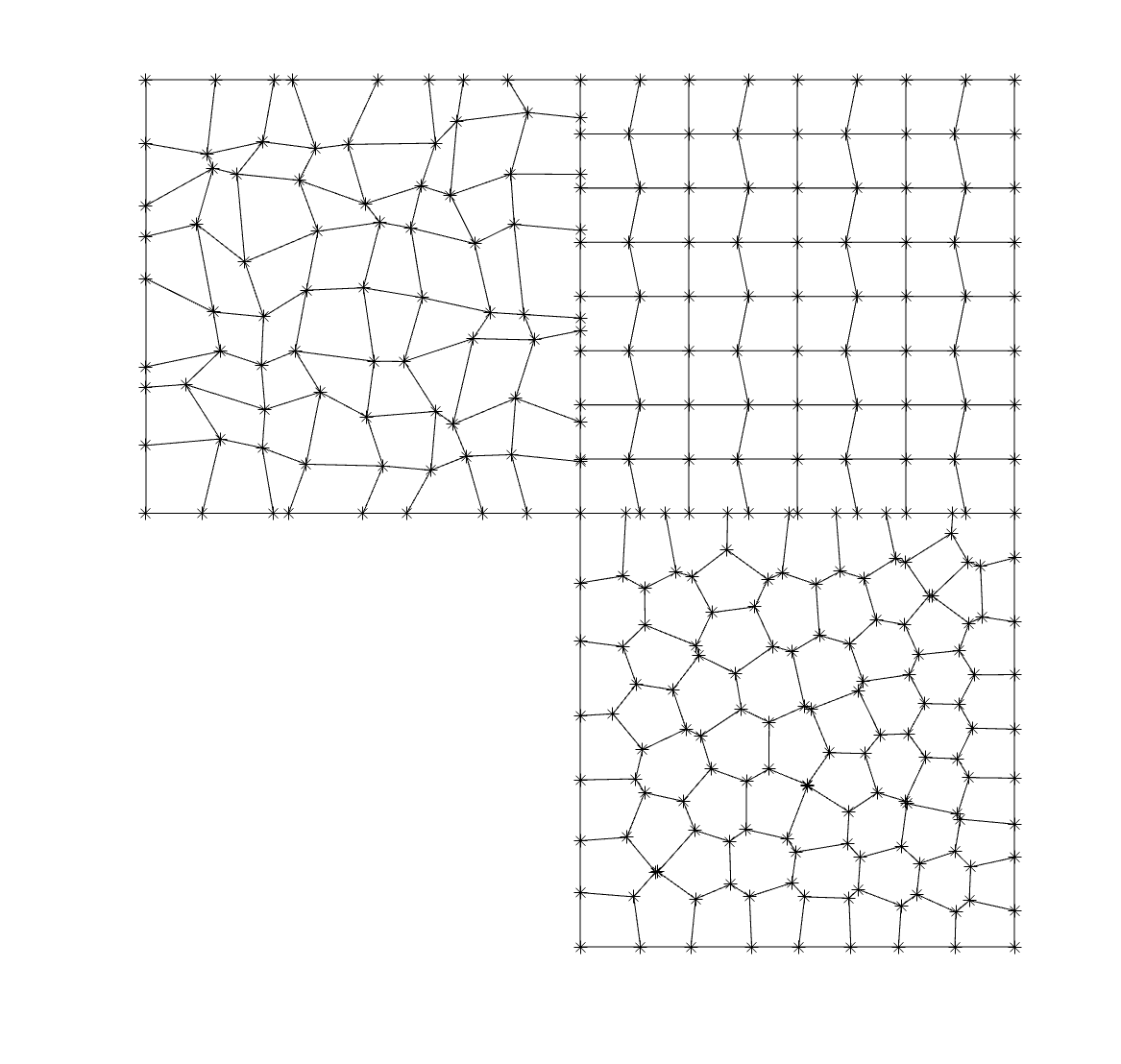}
			\centering\includegraphics[height=4.4cm, width=4.2cm]{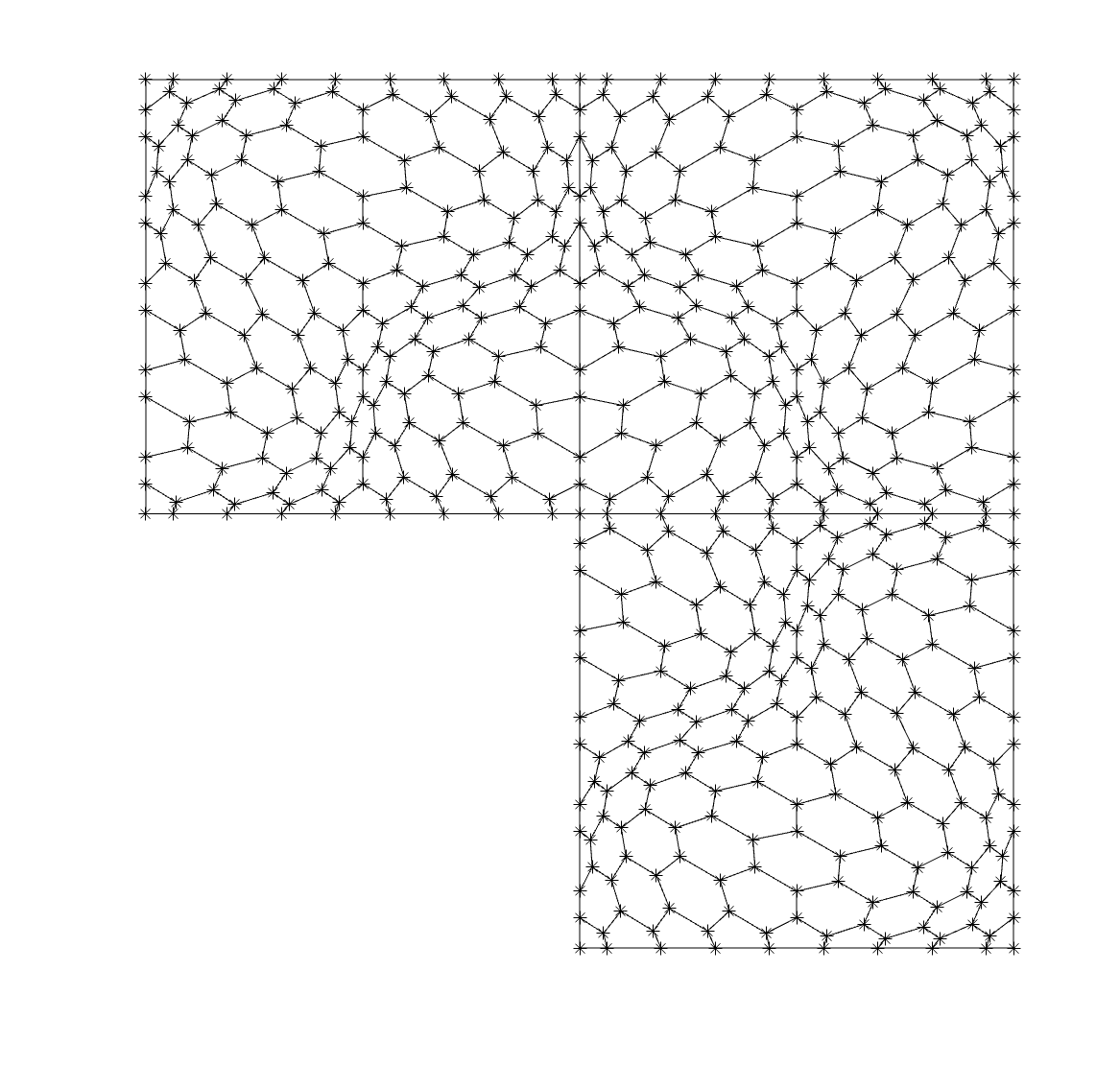}
			\centering\includegraphics[height=4.4cm, width=4.2cm]{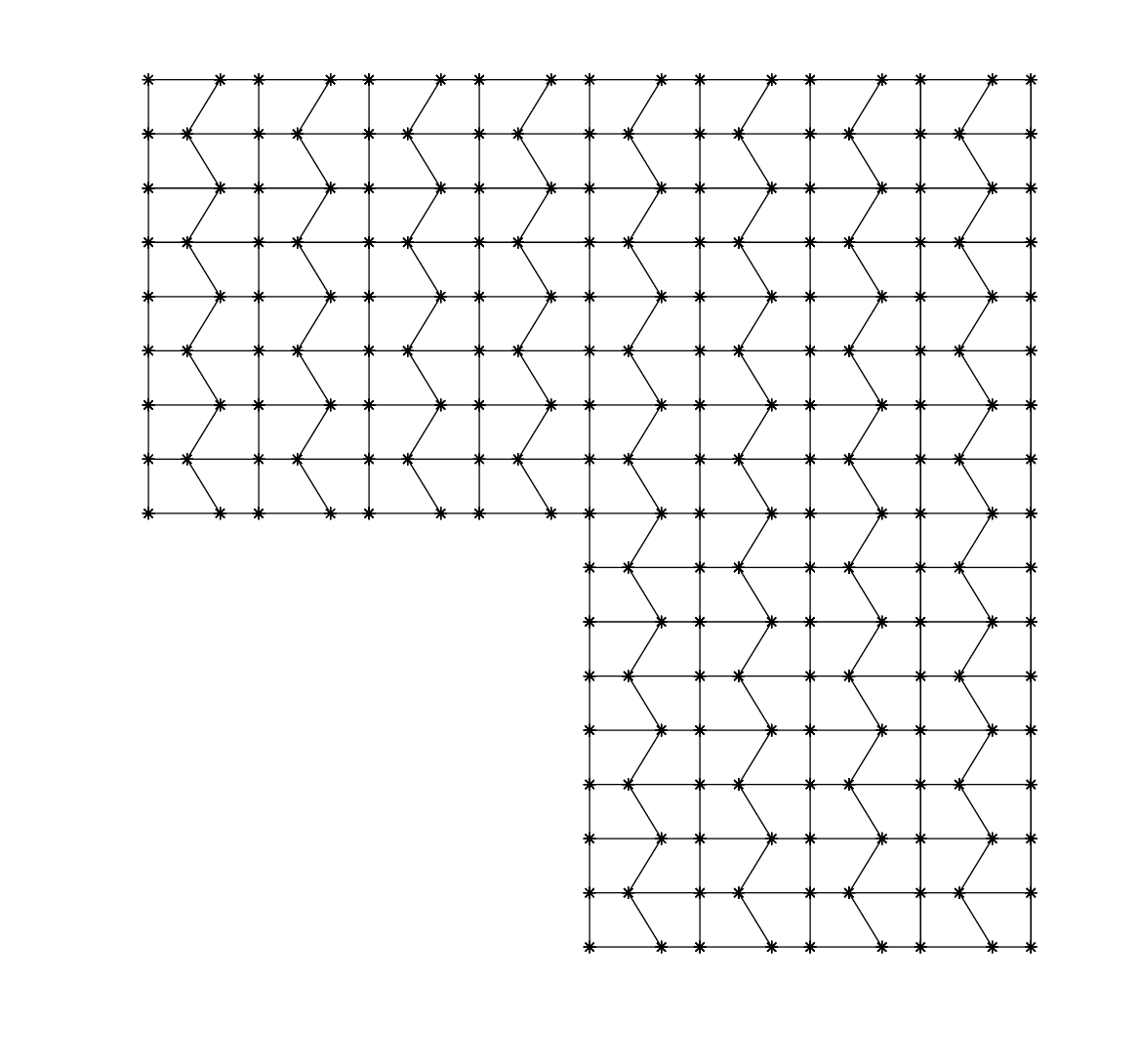}
		\caption{Sample of meshes. Left: $\mathcal{T}_{h}^{5}$ (left),  $\mathcal{T}_{h}^{6}$ (middle),  $\mathcal{T}_{h}^{7}$ (left)  with $N = 8$.}
		\label{fig:meshesL}
	\end{center}
\end{figure}
Owing to the  singularity in this specific geometric configuration, some eigenfunctions of the problem  studied display insufficient smoothness. This, in turn, results in a decrease in the convergence order of the numerical method. Since there is no precise solution for this particular geometry, our results will be compared with the extrapolated eigenvalue. In the following tables, we present the results associated with this problem configuration for different mesh families. (see Tables \ref{tabla5}, \ref{tabla6} and \ref{tabla7}).
\begin{table}[h]
\caption{ The lowest computed eigenvalues $\l_{h,i}$, $1\leq i\leq 4$ for  $\mathcal{T}_h^5$.}
\label{tabla5}
\begin{center}
\begin{tabular}{|c|c|c|c|c|c|c|c|} \hline
 $\l_{h,i}$ & $N = 9$ & $N = 19$& $N = 27$ & $N = 35$ & $N = 45$ &Order & Extr. \\ \hline 
 $\l_{1,h}$  &2.9865e6 &  3.0055e6 &  3.0106e6  & 3.0129e6  & 3.0141e6 & 1.36 &3.0176e6\\
 $\l_{2,h}$ &  7.2211e6 &  7.2250e6 &  7.2259e6 &  7.2263e6 &  7.2264e6  &1.74 &7.2268e6\\
 $\l_{3,h}$ &  2.0098e7 &  2.0162e7 &  2.0174e7  & 2.0178e7  & 2.0179e7  &2.17 &2.0182e7\\
 $\l_{4,h}$ &  2.0132e7 &  2.0165e7 &  2.0176e7  & 2.0179e7 &  2.0180e7  &1.83 &2.0182e7\\\hline
\end{tabular}
\end{center}
\end{table}

\begin{table}[h]
\caption{ The lowest computed eigenvalues $\l_{h,i}$, $1\leq i\leq 4$ for  $\mathcal{T}_h^6$.}
\label{tabla6}
\begin{center}
\begin{tabular}{|c|c|c|c|c|c|c|c|} \hline
 $\l_{h,i}$ & $N = 8$ & $N = 16$& $N = 32$ & $N = 64$  &Order & Extr. \\ \hline 
$\l_{1,h}$  & 2.9433e6 &  2.9911e6  & 3.0070e6 &  3.0135e6  & 1.52   & 3.0163e6\\
$\l_{2,h}$  &   7.2062e6 &  7.2223e6  & 7.2256e6 &  7.2265e6 &2.22       &7.2266e6\\
$\l_{3,h}$  &   1.9751e7 &  2.0095e7  & 2.0163e7  & 2.0178e7 & 2.32 &2.0181e7\\
$\l_{4,h}$  &   1.9763e7 &  2.0101e7 &  2.0163e7  & 2.0178e7 & 2.39 &2.0180e7\\\hline
\end{tabular}
\end{center}
\end{table}

\begin{table}[h]
\caption{ The lowest computed eigenvalues $\l_{h,i}$, $1\leq i\leq 4$ for  $\mathcal{T}_h^7$.}
\label{tabla7}
\begin{center}
\begin{tabular}{|c|c|c|c|c|c|c|c|} \hline
 $\l_{h,i}$ & N = 8 & N = 16& N = 32 & N = 64  &Order & Extr. \\ \hline 
$\l_{1,h}$  & 2.9478e6 &  2.9914e6 &  3.00756 &  3.0136e6 & 1.43&  3.0171e6\\
$\l_{2,h}$  &    7.2015e6 &  7.2204e6   &7.2249e6  & 7.2263e6 &  2.02 &7.2266e6\\
$\l_{3,h}$  &    1.9738e7 &  2.0088e7&   2.0159e7&   2.0177e7 & 2.26&  2.0180e7\\
$\l_{4,h}$  &    1.9814e7 &  2.0105e7 &  2.0162e7&   2.0177e7 &2.29 &  2.0179e7\\\hline
\end{tabular}
\end{center}
\end{table}
As can be seen from Tables \ref{tabla5}, \ref{tabla6}, and \ref{tabla7} the order of the eigenvalues decays for the first eigenvalue; however, the other eigenvalues maintain the order 2. This is the expected behavior, which holds for any polygonal mesh under consideration. We compare the results for each mesh with the last column on the aforementioned tables where we report the extrapolated values obtained with the least square fitting, noting that are close to this extrapolated values. Finally, in Figure \ref{fig:p&u_L} we present plots of the pressure and displacement for the  three smallest eigenvalues, obtained with the different mesh families.
\begin{figure}[h]
	\begin{center}
			\centering\includegraphics[height=4.0cm, width=4.0cm]{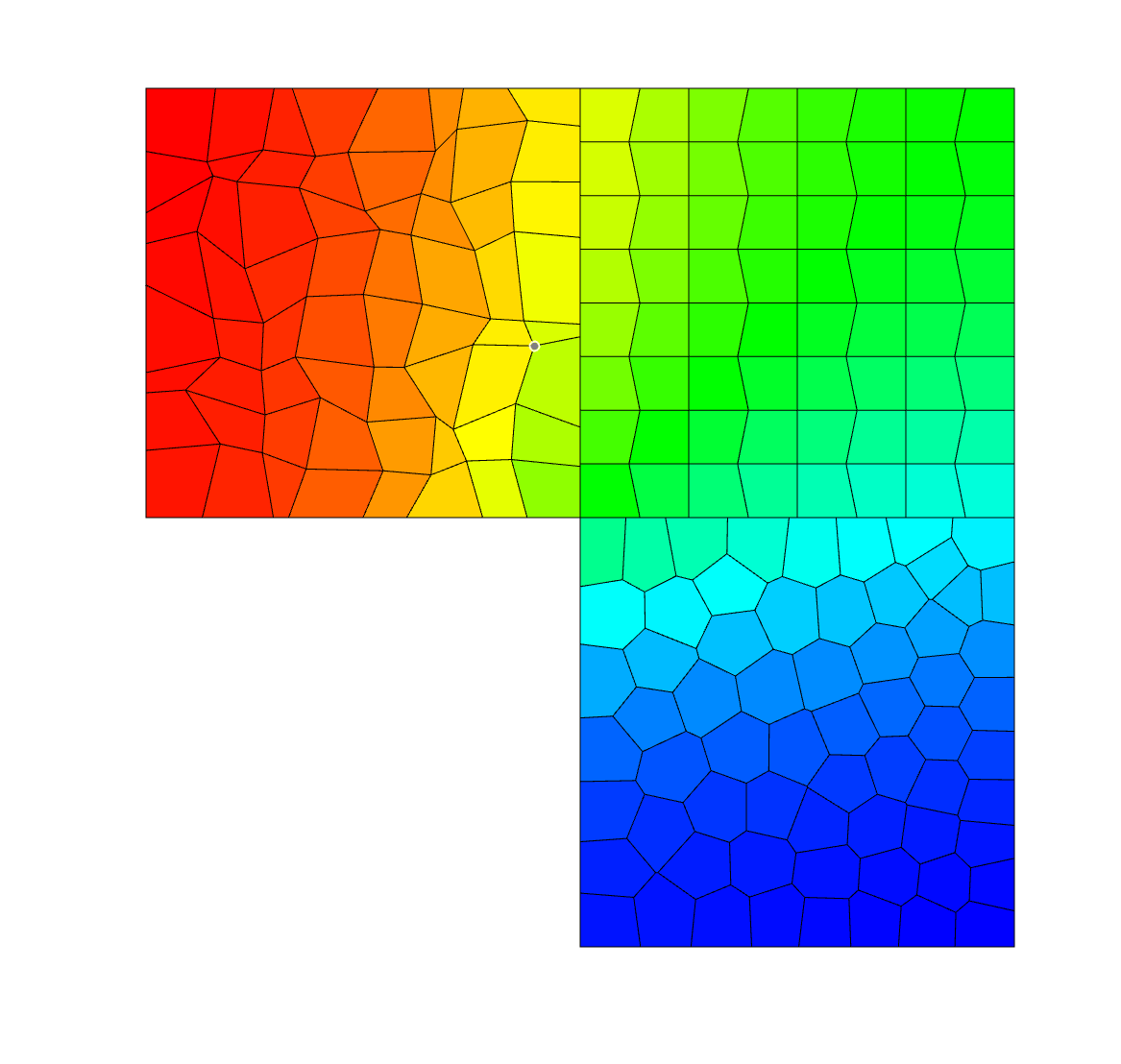}
			\centering\includegraphics[height=4.0cm, width=4.0cm]{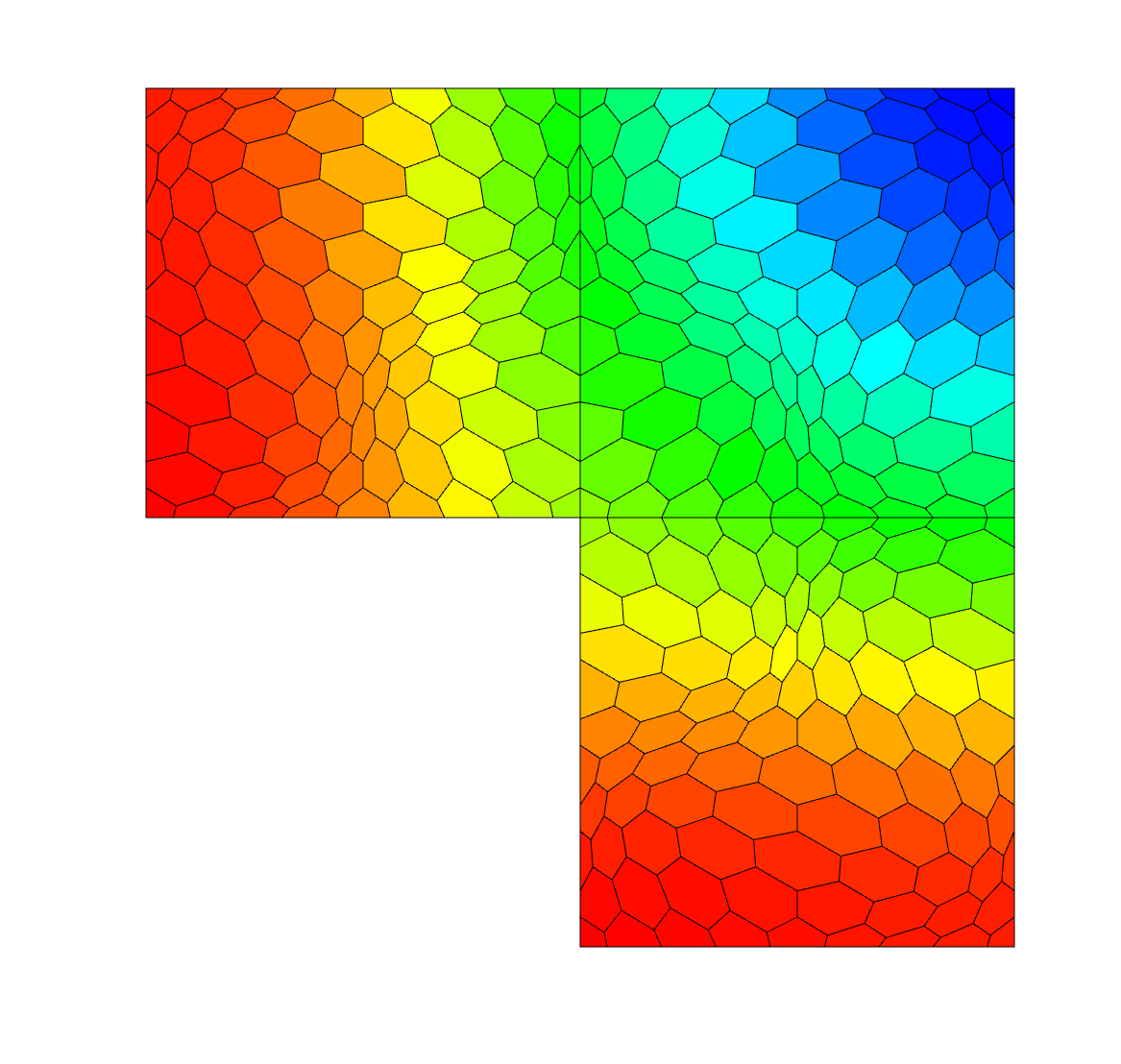}
			\centering\includegraphics[height=4.0cm, width=4.0cm]{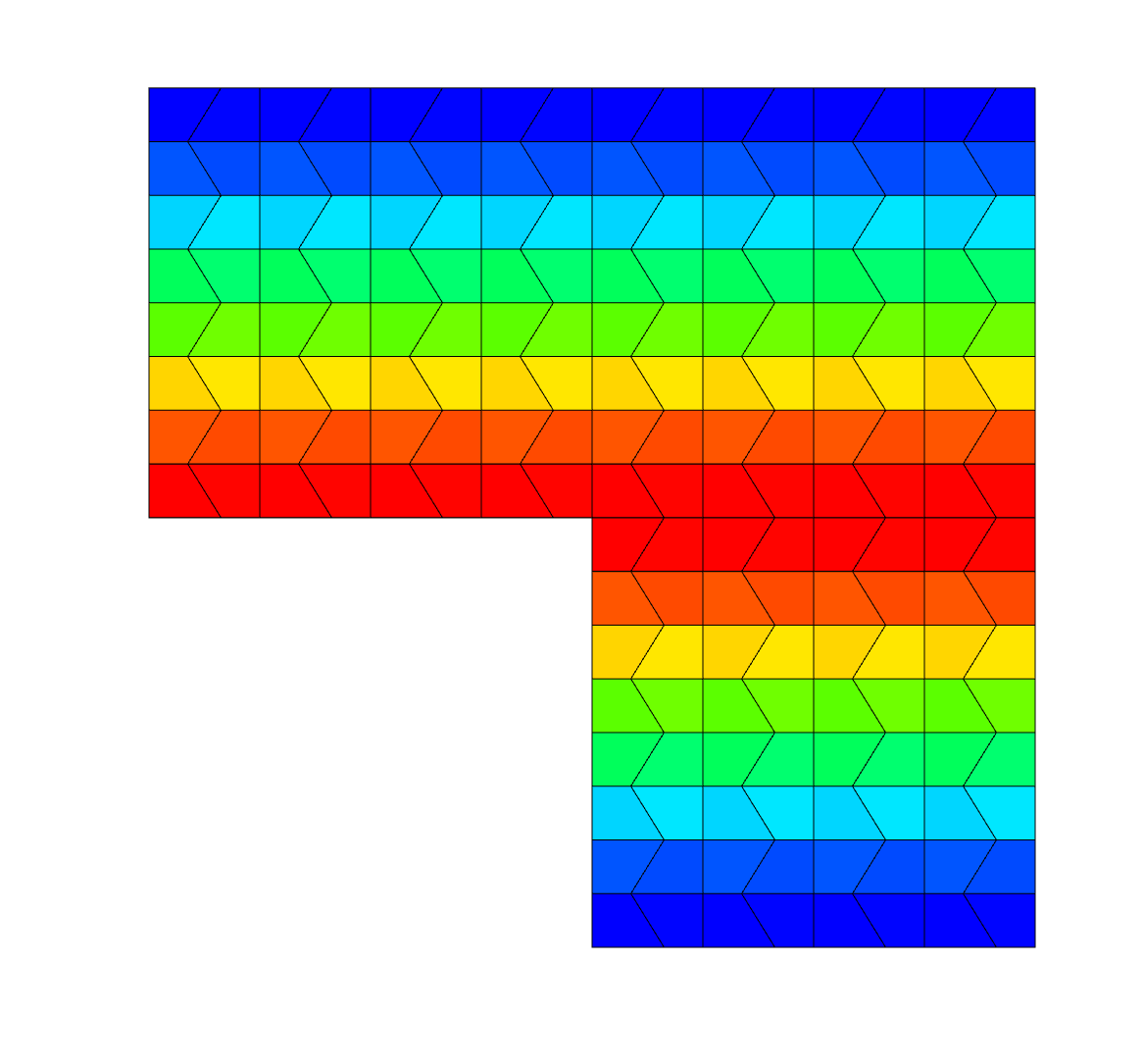}\\
			\centering\includegraphics[height=4.0cm, width=4.0cm]{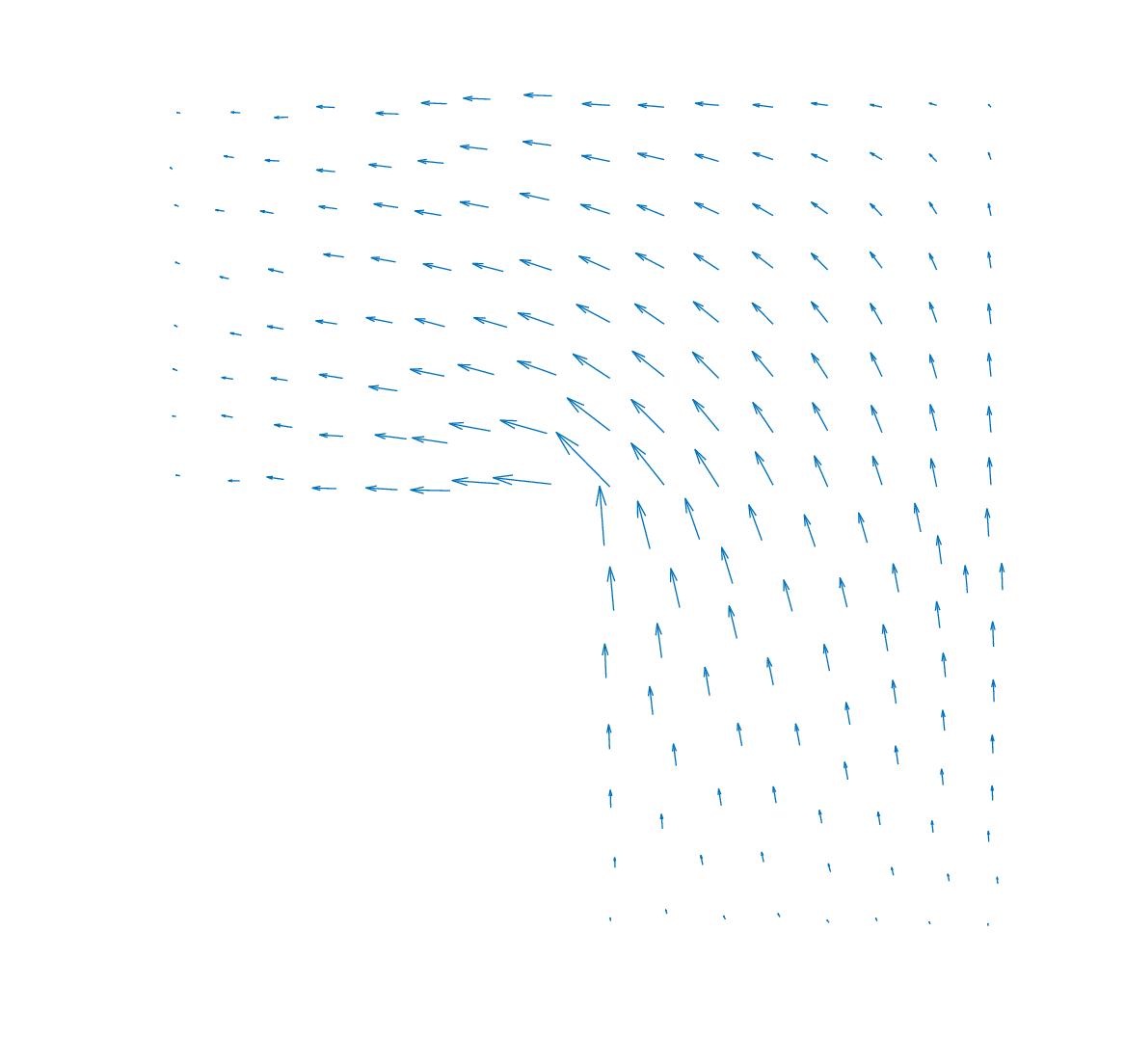}
			\centering\includegraphics[height=4.0cm, width=4.0cm]{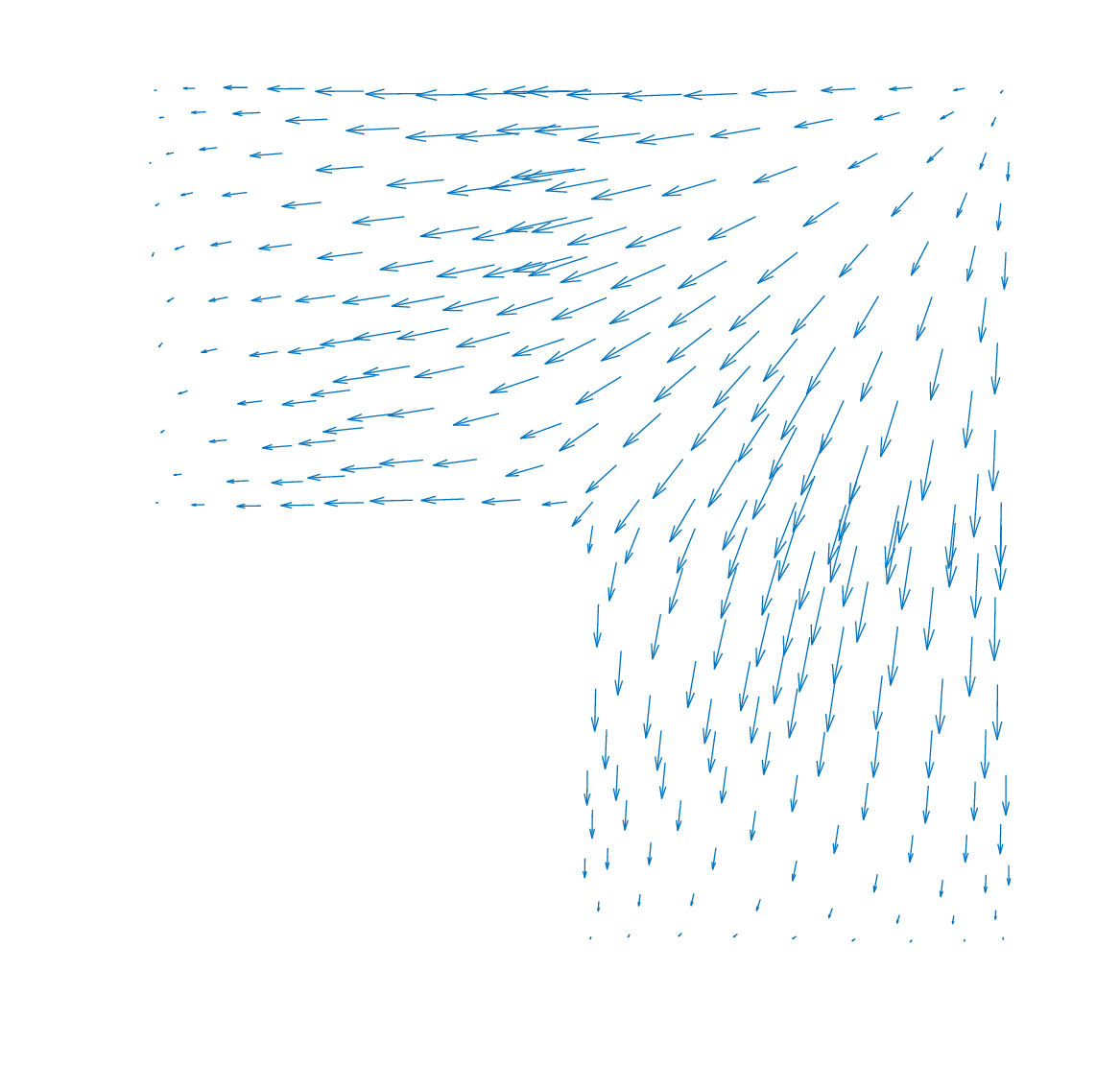}
			\centering\includegraphics[height=4.0cm, width=4.0cm]{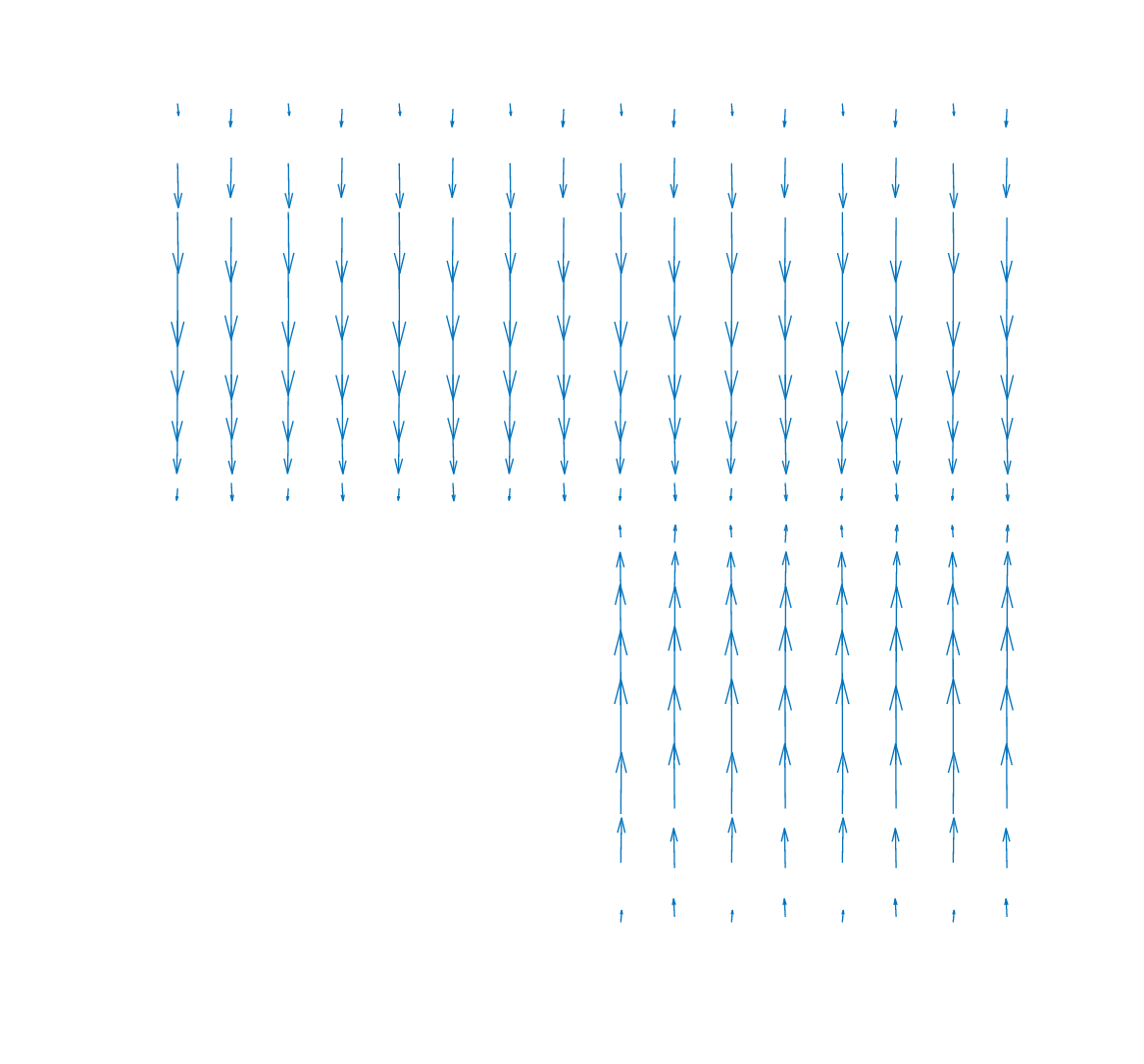}
		\caption{First, second and third  eigenfunctions on the rectangular acoustic cavity with air and water interaction  and the corresponding displacements obtained for this test. Top row: $p_{1,h}$, $p_{2,h}$ and $p_{3,h}$; bottom row: corresponding displacement fields $u_{1,h}$, $u_{2,h}$ and $u_{3,h}$ for  different mesh families.}
		\label{fig:p&u_L}
	\end{center}
\end{figure}
\section{Conclusions}
We have analyzed a non-conforming virtual element method for the acoustic problem with the pure pressure formulation. The method shows accuracy on the computation of the spectrum and, as is expected, is capable to approximate this spectrum with no spurious eigenvalues, after a correct choice of the stabilization parameter. In fact, we see on the numerical tests that is sufficient to take this parameter equal to one to perform safely the method. as the literature indicates. The convergence rates are the theoretically expected for the lowest order of approximation, independently of the polygonal mesh under use. As a final comment, it is important to note that all the of this paper results  can be naturally extended  t three dimensional case.
\bibliographystyle{siamplain}
\bibliography{AmLeRi_7}

\begin{thebibliography}{10}

\bibitem{ADAK}
{\sc D.~Adak, D.~Mora, and I.~Velásquez}, {\em Nonconforming virtual element
  discretization for the transmission eigenvalue problem}, Comput. Math. Appl.,
  152 (2023), pp.~250--267,
  \url{https://doi.org/https://doi.org/10.1016/j.camwa.2023.10.032}.

\bibitem{AABMR13}
{\sc B.~Ahmad, A.~Alsaedi, F.~Brezzi, L.~D. Marini, and A.~Russo}, {\em
  Equivalent projectors for virtual element methods}, Comput. Math. Appl., 66
  (2013), pp.~376--391, \url{https://doi.org/10.1016/j.camwa.2013.05.015}.

\bibitem{amigo2023vem}
{\sc D.~Amigo, F.~Lepe, and G.~Rivera}, {\em Vem allowing small edges for the
  acoustic problem}, 2023, \url{https://arxiv.org/abs/2310.07955}.

\bibitem{MR4658607}
{\sc D.~Amigo, F.~Lepe, and G.~Rivera}, {\em A virtual element method for the
  elasticity spectral problem allowing for small edges}, J. Sci. Comput., 97
  (2023), pp.~Paper No. 54, 29,
  \url{https://doi.org/10.1007/s10915-023-02372-6}.

\bibitem{MR3507277}
{\sc B.~Ayuso~de Dios, K.~Lipnikov, and G.~Manzini}, {\em The nonconforming
  virtual element method}, ESAIM Math. Model. Numer. Anal., 50 (2016),
  pp.~879--904, \url{https://doi.org/10.1051/m2an/2015090}.

\bibitem{BO}
{\sc I.~Babu\v{s}ka and J.~Osborn}, {\em Eigenvalue problems}, vol.~II of
  Handb. Numer. Anal., North-Holland, Amsterdam, 1991.

\bibitem{BBCMMR2013}
{\sc L.~Beir\~{a}o~da Veiga, F.~Brezzi, A.~Cangiani, G.~Manzini, L.~D. Marini,
  and A.~Russo}, {\em Basic principles of virtual element methods}, Math.
  Models Methods Appl. Sci., 23 (2013), pp.~199--214,
  \url{https://doi.org/10.1142/S0218202512500492}.

\bibitem{BLR2017}
{\sc L.~Beir\~{a}o~da Veiga, C.~Lovadina, and A.~Russo}, {\em Stability
  analysis for the virtual element method}, Math. Models Methods Appl. Sci., 27
  (2017), pp.~2557--2594, \url{https://doi.org/10.1142/S021820251750052X}.

\bibitem{BMRR}
{\sc L.~Beir\~{a}o~da Veiga, D.~Mora, G.~Rivera, and R.~Rodr\'{\i}guez}, {\em A
  virtual element method for the acoustic vibration problem}, Numer. Math., 136
  (2017), pp.~725--763, \url{https://doi.org/10.1007/s00211-016-0855-5}.

\bibitem{MR1342293}
{\sc A.~Berm\'{u}dez, R.~Dur\'{a}n, M.~A. Muschietti, R.~Rodr\'{\i}guez, and
  J.~Solomin}, {\em Finite element vibration analysis of fluid-solid systems
  without spurious modes}, SIAM J. Numer. Anal., 32 (1995), pp.~1280--1295,
  \url{https://doi.org/10.1137/0732059}.

\bibitem{MR1770352}
{\sc A.~Berm\'{u}dez, R.~G. Dur\'{a}n, R.~Rodr\'{\i}guez, and J.~Solomin}, {\em
  Finite element analysis of a quadratic eigenvalue problem arising in
  dissipative acoustics}, SIAM J. Numer. Anal., 38 (2000), pp.~267--291,
  \url{https://doi.org/10.1137/S0036142999360160}.

\bibitem{MR1993937}
{\sc A.~Berm\'{u}dez, P.~Gamallo, L.~Hervella-Nieto, and R.~Rodr\'{\i}guez},
  {\em Finite element analysis of pressure formulation of the elastoacoustic
  problem}, Numer. Math., 95 (2003), pp.~29--51,
  \url{https://doi.org/10.1007/s00211-002-0414-0}.

\bibitem{MR2086168}
{\sc A.~Berm\'{u}dez, P.~Gamallo, and R.~Rodr\'{\i}guez}, {\em Finite element
  methods in local active control of sound}, SIAM J. Control Optim., 43 (2004),
  pp.~437--465, \url{https://doi.org/10.1137/S0363012903431785}.

\bibitem{DNR1}
{\sc J.~Descloux, N.~Nassif, and J.~Rappaz}, {\em On spectral approximation.
  part 1. the problem of convergence}, ESAIM: Mathematical Modelling and
  Numerical Analysis-Mod{\'e}lisation Math{\'e}matique et Analyse
  Num{\'e}rique, 12 (1978), pp.~97--112.

\bibitem{DNR2}
{\sc J.~Descloux, N.~Nassif, and J.~Rappaz}, {\em On spectral approximation.
  part 2. error estimates for the galerkin method}, RAIRO. Analyse
  num{\'e}rique, 12 (1978), pp.~113--119.

\bibitem{GMV2018}
{\sc F.~Gardini, G.~Manzini, and G.~Vacca}, {\em The nonconforming virtual
  element method for eigenvalue problems}, ESAIM Math. Model. Numer. Anal., 53
  (2019), pp.~749--774, \url{https://doi.org/10.1051/m2an/2018074}.

\bibitem{MR3867390}
{\sc F.~Gardini and G.~Vacca}, {\em Virtual element method for second-order
  elliptic eigenvalue problems}, IMA J. Numer. Anal., 38 (2018),
  pp.~2026--2054, \url{https://doi.org/10.1093/imanum/drx063}.

\bibitem{MR0775683}
{\sc P.~Grisvard}, {\em Elliptic problems in nonsmooth domains}, vol.~24 of
  Monographs and Studies in Mathematics, Pitman (Advanced Publishing Program),
  Boston, MA, 1985.

\bibitem{MR0203473}
{\sc T.~Kato}, {\em Perturbation theory for linear operators}, vol.~Band 132 of
  Grundlehren der Mathematischen Wissenschaften, Springer-Verlag, Berlin-New
  York, second~ed., 1976.

\bibitem{MR4077220}
{\sc F.~Lepe and D.~Mora}, {\em Symmetric and nonsymmetric discontinuous
  {G}alerkin methods for a pseudostress formulation of the {S}tokes spectral
  problem}, SIAM J. Sci. Comput., 42 (2020), pp.~A698--A722,
  \url{https://doi.org/10.1137/19M1259535},
  \url{https://doi.org/10.1137/19M1259535}.

\bibitem{MR4284360}
{\sc F.~Lepe, D.~Mora, G.~Rivera, and I.~Vel\'{a}squez}, {\em A virtual element
  method for the {S}teklov eigenvalue problem allowing small edges}, J. Sci.
  Comput., 88 (2021), pp.~Paper No. 44, 21,
  \url{https://doi.org/10.1007/s10915-021-01555-3}.

\bibitem{MR4550402}
{\sc F.~Lepe, D.~Mora, G.~Rivera, and I.~Vel\'{a}squez}, {\em A posteriori
  virtual element method for the acoustic vibration problem}, Adv. Comput.
  Math., 49 (2023), pp.~Paper No. 10, 29,
  \url{https://doi.org/10.1007/s10444-022-10003-1}.

\bibitem{MR4229296}
{\sc F.~Lepe and G.~Rivera}, {\em A virtual element approximation for the
  pseudostress formulation of the {S}tokes eigenvalue problem}, Comput. Methods
  Appl. Mech. Engrg., 379 (2021), pp.~Paper No. 113753, 21,
  \url{https://doi.org/10.1016/j.cma.2021.113753}.

\bibitem{MR4050542}
{\sc D.~Mora and G.~Rivera}, {\em {\it {A} priori} and {\it a posteriori} error
  estimates for a virtual element spectral analysis for the elasticity
  equations}, IMA J. Numer. Anal., 40 (2020), pp.~322--357,
  \url{https://doi.org/10.1093/imanum/dry063}.

\bibitem{MR3340705}
{\sc D.~Mora, G.~Rivera, and R.~Rodr\'{\i}guez}, {\em A virtual element method
  for the {S}teklov eigenvalue problem}, Math. Models Methods Appl. Sci., 25
  (2015), pp.~1421--1445, \url{https://doi.org/10.1142/S0218202515500372}.

\bibitem{HYBi2017NonConf}
{\sc Y.~Yang, J.~Han, and H.~Bi}, {\em Non-conforming finite element methods
  for transmission eigenvalue problem}, Comput. Methods Appl. Mech. Engrg., 307
  (2016), pp.~144--163.

\end{thebibliography}
\end{document}